\documentclass[11pt, a4paper]{article}
\usepackage{amssymb,amsmath,amsthm}
\usepackage{fancyhdr}
\usepackage[british]{babel}
\usepackage{color}
\usepackage[title]{appendix}
\usepackage{amsmath,amsthm}

\usepackage{enumerate,enumitem}
\usepackage{enumitem,verbatim}

\usepackage{ulem}
\usepackage{hyperref}
\usepackage[margin=2.5cm]{geometry}
\usepackage{lineno}
\usepackage[ruled]{algorithm2e}
\usepackage{todonotes}

\newtheorem{theorem}{Theorem}[section]
\newtheorem{prop}[theorem]{Proposition}
\newtheorem{lemma}[theorem]{Lemma}

\newtheorem{conjecture}[theorem]{Conjecture}

\theoremstyle{definition}

\newtheorem{question}{Question}

\newtheorem{defn}[theorem]{Definition}

\newtheorem{claim}[theorem]{Claim}
\newtheorem{fact}[theorem]{Fact}

\newcommand{\btheorem}{\begin{theorem}}
\newcommand{\etheorem}{\end{theorem}}
\newcommand{\bconjecture}{\begin{conjecture}}
\newcommand{\econjecture}{\end{conjecture}}
\newcommand{\bproposition}{\begin{proposition}}
\newcommand{\eproposition}{\end{proposition}}
\newcommand{\bdefinition}{\begin{definition}}
\newcommand{\edefinition}{\end{definition}}
\newcommand{\bcorollary}{\begin{corollary}}
\newcommand{\ecorollary}{\end{corollary}}
\newcommand{\bproof}{\begin{proof}}
\newcommand{\eproof}{\end{proof}}
\newcommand{\bclaim}{\begin{claim}}
\newcommand{\eclaim}{\end{claim}}
\newcommand{\bquestion}{\begin{question}}
\newcommand{\equestion}{\end{question}}
\newcommand{\bfact}{\begin{fact}}
\newcommand{\efact}{\end{fact}}
\newcommand{\bremark}{\begin{remark}}
\newcommand{\eremark}{\end{remark}}
\newcommand{\eexample}{\end{example}}
\newcommand{\bexample}{\begin{example}}

\newcommand{\elemma}{\end{lemma}}
\newcommand{\blemma}{\begin{lemma}}

\newcommand{\be}{\beta}

\newcommand{\N}{\mathbb{N}}

\newcommand{\eps}{\varepsilon}

\newcommand{\de}{\delta}

\title{Packing tetrahedrons in edge-weighted graphs }

\author{
Wanting Sun\thanks{Data Science Institute, Shandong University, Jinan, China. Email: {\tt wtsun@sdu.edu.cn}.},
Shunan Wei\thanks{School of Mathematics, Shandong University, Jinan, China. Email: {\tt snwei@mail.sdu.edu.cn}.},
Donglei Yang\thanks{School of Mathematics, Shandong University, Jinan, China. Email: {\tt dlyang@sdu.edu.cn}.}}

\date{\today}

\begin{document}
\maketitle
\linespread{1.2}

\begin{abstract}
 We prove that for all $\mu>0, t\in (0,1)$ and sufficiently large $n\in 4\N$, if $G$ is an edge-weighted complete graph on $n$ vertices with a weight function $w: E(G)\rightarrow [0,1]$ and the minimum weighted degree $\de^w(G)\ge (\tfrac{1+3t}{4}+\mu)n$, then $G$ contains a $K_4$-factor where each copy of $K_4$ has total weight more than $6t$. This confirms a conjecture of Balogh--Kemkes--Lee--Young for the tetrahedron case. 
\end{abstract}


\section{Introduction}

Given two graphs $H$ and $G$, an $H$-tiling of $G$ is a collection of vertex-disjoint copies of $H$ in the graph $G$. A perfect $H$-tiling of $G$, or an $H$-factor, is an $H$-tiling that covers all vertices of $G$.
A fundamental research topic in Extremal Combinatorics is to determine sufficient minimum degree conditions forcing certain spanning structures, such as perfect matchings (i.e.,  $K_2$-factors), Hamilton cycles, $H$-factors, etc. {Textbook results of Hall and Tutte give a sufficient condition for the existence of a perfect matching.} A cornerstone theorem of Dirac~\cite{Dirac} states that every graph $G$ on $n\ge 3$ vertices with minimum degree $\delta(G)\ge \frac{n}{2}$ contains a Hamilton cycle, in particular if $n$ is even then $G$ has a perfect matching. The celebrated theorem of Hajnal and Szemer\'edi~\cite{HajnalSz} gives a best possible condition on the minimum degree which  guarantees the existence of a $K_r$-factor in a graph (the $r=3$ case was previously obtained by Corr\'adi and Hajnal~\cite{Corradi}). 

\begin{theorem}[Hajnal-Szemer\'edi theorem~\cite{HajnalSz}]\label{thm:HS}
Let $G$ be an $n$-vertex graph with $n\in r\mathbb{N}$. If $\delta(G)\ge \left(1-\frac{1}{r}\right)n$, then $G$ contains a $K_r$-factor.
\end{theorem}

A short proof was later provided by Kierstead and Kostochka~\cite{Kiersteadk2008}. 
For an arbitrary graph $H$, K\"uhn and Osthus~\cite{Kuhn2009} determined the minimum degree condition that ensures an 
$H$-factor. 
  For more results concerning different graphs $H$, see the nice survey
 by K\"{u}hn and Osthus~\cite{Kuhnsurvey2009}.

Although the theorem of Hajnal-Szemer\'edi was established decades ago, recent years have witnessed remarkable progress in its extensions and variants, such as partite graphs, directed graphs and hypergraphs.
In this paper, we primarily focus on a weighted variant of the Hajnal-Szemer\'edi theorem, specifically in the setting of edge-weighted graphs.

\subsection{Weighted Hajnal--Szemer\'{e}di}

Let $G=(V,E,w)$ be an \textit{edge-weighted graph} on the vertex set $V$ together with an edge weighting $w: E\rightarrow [0,1]$. 
Balogh, Kemkes, Lee and Young~\cite{BKLY}  proposed investigating a variant of Hajnal-Szemer\'{e}di Theorem in edge-weighted graphs, asking for the minimum (weighted) degree thresholds for perfect heavy $K_r$-tilings. To illustrate this, given an \textit{edge-weighted graph} $G=(V,E,w)$, we define the \textit{weighted degree} of any vertex $v\in V$ to be $d^w(v):=\sum_{uv\in E}w(uv)$, and denote by $\de^w(G):=\min\{d^w(v):v\in V\}$ the minimum weighted degree of $G$. For a constant $t\in (0,1),$ we say a copy of $K_r$ in $G$ is \textit{$t$-heavy} if the weight sum over all the $\binom{r}{2}$ edges from the copy is strictly more than $\binom{r}{2}t$. 
\begin{conjecture}[\cite{BKLY}]\label{conj: main conj}
For any $\mu>0,\,r\ge 2,\,t\in (0,1)$ and any sufficiently large $n\in r\mathbb{N}$, if $G$ is an $n$-vertex {edge-}weighted complete graph with $\de^w(G)\ge \left(\frac{1}{r}+(1-\frac{1}{r})t+\mu\right)n$, then $G$ contains a $t$-heavy $K_r$-factor.
\end{conjecture}

The minimum weighted degree is asymptotically best possible for every $t\in (0,1)$, as shown by the following example provided in~\cite{BKLY}.
Let $n\in r\mathbb{Z}$ with $n>r$. Let $U\subseteq V(K_n)$ be an arbitrary subset of size exactly $\frac{(r-1)n}{r}+1$ and $W$ be the remaining $\frac{n}{r}-1$ vertices. Consider the weight function $w$ that assigns weight $t$ to edges whose both endpoints are in $U$, and weight $1$ to all other edges. It is routine to check that $\de^w(K_n)=t\cdot(|U|-1)+1\cdot |W|= \left(\frac{1}{r}+(1-\frac{1}{r})t\right)n-1$. Notice that any $r$ vertices in $U$ cannot induce a $t$-heavy $K_r$, that is, every $t$-heavy $K_r$ in $G$ must intersect $W$. As $|W|=\frac{n}{r}-1$, such an edge-weighted complete graph does not contain a $t$-heavy $K_r$-factor. 

Balogh, Kemkes, Lee and Young~\cite{BKLY} determined that the minimum weighted degree condition guaranteeing a $t$-heavy $K_2$-factor is $\frac{1+t}{2}$. Notably, for $r\ge 3$, they proved that when $t$ is sufficiently small, the condition $\delta^w(K_n) \geq \left(\frac{1}{r} + \left(1-\frac{1}{r}\right)t\right)n$ suffices for the existence of a $t$-heavy $K_r$-factor. 
Later on, Balogh, Molla and Sharifzadeh~\cite{BMS} confirmed the triangle case of Conjecture~\ref{conj: main conj} for all $t\in (0,1)$.
Our main result confirms Conjecture~\ref{conj: main conj} for the next open case (i.e., $r=4$). 

\begin{theorem}[Main theorem]\label{main thm}
For all $\mu>0,t\in (0,1)$ and sufficiently large $n\in 4\N$, if $G$ is an $n$-vertex edge-weighted complete graph with $\de^w(G)\ge (\tfrac{1+3t}{4}+\mu)n$, then $G$ contains a $t$-heavy $K_4$-factor.
\end{theorem}

\noindent{\bf Related work.}
Considerable attention has recently been focused on extremal problems concerning edge-weighted graphs, especially in the study of Ramsey-Tur\'an problems.
The \textit{Ramsey-Tur\'an number} $RT_p(n,K_q,\eps n)$ is the maximum number of edges in an $n$-vertex $K_q$-free graph $G$ with $\alpha_p(G)\le\eps n$,  where $\alpha_p(G)$ 
denotes the size 
of a largest $K_p$-free induced subgraph of $G$.  
The \textit{Ramsey-Tur\'an density} is defined as \[
\rho_p(q):=\lim_{\eps\to 0}\lim_{n\to \infty}\frac{RT_p(n,K_q,\eps n)}{\binom{n}{2}}.
\]
This type of problem was initiated in 1969 by S\'{o}s \cite{Sos} and later generalised by Erd\H{o}s and S\'{o}s \cite{Erdos-Sos}, who determined the value of $\rho_2(q)$ for all odd $q$. Since then, determining Ramsey-Tur\'{a}n densities for various parameters $p$ and $q$ has become a central problem in this field.

As is common in such problems, an application of Szemer\'edi's  regularity lemma allows us to reduce the original problems to embedding certain structures in the corresponding edge-weighted reduced graph. 
For instance, Balogh, Brada\v{c} and Lidick\'y~\cite{Balogh2025} recently reformulates the Ramsey-Tur\'{a}n problem as a clique embedding problem in edge-weighted reduced graphs. They further identified a key theoretical barrier: ``The main obstacle for obtaining further upper bounds on Ramsey-Tur\'{a}n densities has been a lack of such weighted Tur\'{a}n results". 
On the other hand, extremal problems in the weighted graph are closely related to multigraph extremal problems. For instance, if we consider a multigraph with multiplicities $1$ and $2$, it can be represented by an edge-weighted graph with the weights of $\frac{1}{2}$ and $1$, respectively. Similar strategies have been adopted in recent developments for example~\cite{Chang2023, Chen2024,gaojiang, Knierim2020, Nenadov2020}. 



\vspace{2mm}
\noindent{\bf Notation.} 
Throughout the paper we follow the standard graph-theoretic notation~\cite{diestel2017graph}. Let $G=(V,E,w)$ be an edge-weighted graph and $U\subseteq V$ be a vertex subset. Denote by $G[U]$ the induced edge-weighted
subgraph of $G$ on $U$. The notation $G-U$ is used to denote the induced weighted  subgraph after removing $U$,
that is $G-U:=G[V\setminus U]$. Given two disjoint vertex subsets $A,B\subseteq V(G)$, we use $G[A,B]$ to denote the edge-weighted bipartite subgraph of $G$ with vertex set $A\cup B$ and edge set consisting of all edges having one endpoint in $A$ and the other in $B$. The total weights of edges in $G[A,B]$ is denoted by $w(A,B)$. When $A$ and $B$ are subgraphs of $G$, we always  abbreviate $G[V(A),V(B)]$ and 
$w(V(A),V(B))$ as $G[A,B]$ and 
$w(A,B)$ respectively; when $A=B$, we abbreviate $w(A,B)$ as $w(A)$.  For any subgraph  defined above, the edge weights are preserved from the original graph 
$G$.

Given a set $U$ 
and an integer $k$, we write
$\binom{U}{k}$ the collection of all 
$k$-element subsets of $U$. For any integers $a\leq b$, define $[a,b]:=\{i\in \mathbb{N}:a\leq i\leq b\}$.
  When we write  $\alpha\ll \beta\ll \gamma$, we always mean that $\alpha, \beta, \gamma$  are constants in $(0,1)$, and $\alpha\ll \beta$ means that there exists $\alpha_0=\alpha_0(\beta)$ such that the subsequent arguments hold for all $0<\alpha\leq \alpha_0$. 
Hierarchies of other lengths are defined analogously.

For simplicity, we will often abbreviate $t$-heavy subgraphs as heavy subgraphs. In what follows, we consider only edge-weighted graphs; thus, the term ``weighted graph" will always refer to an edge-weighted graph unless stated otherwise.

\subsection{Proof overviews}
Our proof primarily adopts the framework of the absorption method proposed by R\"odl, Ruci\'nski and Szemer\'edi~\cite{RRS2009}. 
This method provides an efficient framework for constructing spanning subgraphs.
The absorption method typically decomposes the problem of finding perfect tilings into two parts. The first part involves constructing an absorbing set $A$ in the host weighted complete graph $G:=K_n$, which can `absorb' any small set of `remaining' vertices. The second part is to find an almost perfect $t$-heavy $K_4$-tiling covering almost all vertices in $G-A$.

\vspace{2mm}
\noindent{\bf Absorbing.} 
In order to construct an absorbing set, we adopt a strategy of Nenadov and Pehova~\cite{Nenadov2020} in the edge-weighted setting to transform our construction into finding linearly many vertex-disjoint absorbers for every four vertices. This is usually achieved by the reachability argument as for example in \cite{Lo2015, Nenadov2020}.
In particular, for the triangle case of Conjecture~\ref{conj: main conj}, Balogh, Molla and Sharifzadeh~\cite{BMS} were able to show that if $\de^w(G)\ge (\tfrac{1+2t}{3}+o(1))n$, then every two vertices in $V(G)$ are $1$-reachable (to be defined later), and this together with a standard construction (see Lemma~\ref{lem: absorber lem}) guarantees $\Omega(n)$ vertex-disjoint absorbers for every three vertices. 
However, it is quite difficult to extend this to the tetrahedron case under a much weaker degree condition $\de^w(G)\ge (\tfrac{1+3t}{4}+o(1))n$. 
To address this, we establish a technical lemma (see Lemma~\ref{lem: two reachable from three}) which tells that
\begin{quote}
\textit{for every three vertices, there exist two vertices that are $1$-reachable} .
\end{quote}
This intuitively means that almost all vertices can be $1$-reachable to many other vertices in $V(G)$. In this case, our constructions carry on with a partition $V(G)=B\cup U$, where $|B|=o(n)$ and every pair of vertices in $U$ are $32$-reachable (see Lemma \ref{main lemma}), and this is the bulk of our proof. 

For the proof of Lemma \ref{main lemma}, we employ a merging process by combining the transferral detecting (see Lemma~\ref{lem: merge lem}) with many new geometric perspectives on heavy tetrahedrons. Roughly speaking, we shall first show that $V(G)$ can be partitioned into vertex subsets $V_1, V_2$, such that every pair of vertices in the same subset are $4$-reachable (see Lemma \ref{lem: partition lem}). It remains to show that vertices from distinct parts are reachable, so as to merge $V_1$ and $V_2$ into a closed set as desired. Lemma~\ref{lem: merge lem} reduces this to detecting the existence of transferrals (see Definition~\ref{lattice}), that is, to find two different classes of heavy tetrahedrons with prescribed distributions (in $V_1$ and $V_2$). In doing this, we first utilize a multicolored version of the regularity method and use an embedding lemma (see Lemma \ref{lem: embedding lem}) to embed linearly many disjoint heavy tetrahedrons of certain types. An advantage for this is that it suffices to find one (instead of many) heavy tetrahedron as required in the corresponding weighted reduced graph. We also bring some insights of heavy tetrahedrons (see Proposition~\ref{prop: heavy}). Note that the $1$-reachability between any two vertices means the existence of two heavy tetrahedrons sharing a common triangular face. In this way, Proposition~\ref{prop: heavy} tells that every heavy tetrahedron constructed as in \ref{per1} contains two heavy triangles sharing an edge, while Proposition~\ref{prop: heavy triangle} \ref{quarter2} enables you to extend them to two heavy tetrahedrons sharing a common triangle. This essentially guides us to find transferrals by checking various tetrahedrons as in \ref{per1} (in the reduced graph) which intersect with both $V_1$ and $V_2$. We remark that our new perspectives (see Proposition~\ref{prop: heavy} and the remark afterwards) may shed light on a full resolution of Conjecture~\ref{conj: main conj}.

\vspace{2mm}
\noindent{\bf Almost cover.} 
Following the previous framework in~\cite{BKLY} for the triangle case, our construction of almost perfect heavy $K_4$-tiling becomes surprisingly challenging and involved. This is where our new perspectives (in Proposition~\ref{prop: heavy}) on heavy tetrahedrons come into play. We first greedily select a maximal collection of vertex-disjoint heavy $K_4$s, denoted as $\mathcal{R}$. Here, the key difference is that we employ a technical selection as follows:  
\begin{quote}
\textit{every copy of $K_4$ from $\mathcal{R}$ must contain at least  two heavy triangles and two heavy edges. }
\end{quote}
For the remaining vertices, we greedily choose a maximal family $\mathcal{T}$ of vertex-disjoint heavy triangles, followed by a maximal matching $M$ of heavy edges in what remains. Throughout this process, the selection is made to maximize 
the tuple $(|\mathcal{R}|,  |\mathcal{T}|,|M|,\rho)$ lexicographically, where $\rho$ denotes the total weights of all heavy $K_4$ in $\mathcal{R}$.  
By viewing $K_4$ as a tetrahedron and careful analysis on heavy faces of each tetrahedron in $\mathcal{R}$, we show that actually $\mathcal{R}$ covers almost all vertices, as otherwise we can augment the tuple $(|\mathcal{R}|,  |\mathcal{T}|,|M|,\rho)$ in various cases, yielding a contradiction.

\vspace{2mm}
\noindent{\bf Organization.} The rest of the paper is organized as follows. In the next section, we shall
introduce an absorbing lemma (Lemma \ref{abs}) and an almost perfect tiling lemma (Lemma \ref{almost}) for proving Theorem \ref{main thm}. In Section 3, we shall prove Lemma \ref{almost}. Section 4 will briefly present some necessary results and tools to introduce the lattice-based absorbing method and how they come together to prove Lemma \ref{abs}.

\section{Proof of Theorem \ref{main thm} and Preliminaries}
\subsection{Proof of Theorem \ref{main thm}}

We will start with the notion of absorbers needed in our proof.
\begin{defn}[Absorber]\label{defabs}
Let $H$ be an $h$-vertex graph and $G$ be an $n$-vertex weighted  graph.
For any $h$-set $S\subseteq V(G)$, any integer $k$ and $t\in (0,1)$, we say that a set $A_S\subseteq V(G)\setminus S$ is an \textit{$(H,k)$-absorber} for $S$  if $|A_S|\le h^2k$ and both $G[A_S]$ and $G[A_S\cup S]$ contain $t$-heavy $H$-factors.
\end{defn}

\begin{defn}[Absorbing set]
Let $H$ be an $h$-vertex graph, $G$ be an $n$-vertex weighted  graph, $\xi$ be a positive  constant and $t\in (0,1)$.
A set $A\subseteq V(G)$ is called an $(H,\xi)$-\textit{absorbing set} (for $V(G)$) if for any set $R\subseteq V(G)\setminus A$ with $|R|\le \xi n$ and $|A\cup R|\in h\mathbb{N}$, then graph $G[A\cup R]$ contains a $t$-heavy $H$-factor.
\end{defn}
Our first task is to construct an absorbing set. 
\begin{lemma}[Absorbing lemma]\label{abs}
For any constants $\mu>0,\,t\in(0,1)$, there exist $\gamma,\,\xi>0$ and $n_0\in \mathbb{N}$ such that the following statement holds. 
If $G$ is a weighted  complete graph with $n\geq n_0$ vertices and $\de^w(G)\ge (\tfrac{1+3t}{4}+\mu)n$, then $G$ contains a $(K_4,  \xi)$-absorbing set of size at most $\gamma n$.
\end{lemma}
\medskip

The second major task is to find a heavy $K_4$-tiling that covers almost all vertices. 

\begin{lemma}[Almost cover]\label{almost}
For any $\mu>0$ and $t\in(0,1)$, there exists $n_0\in \mathbb{N}$ such that the following statement holds. If $G$ is a weighted complete graph with $n\geq n_0$ vertices and $\de^w(G)\ge (\tfrac{1+3t}{4}+\mu)n$, then $G$ contains a $t$-heavy $K_4$-tiling covering all but at most $9+\frac{3}{\mu}$ vertices.
\end{lemma}

We end this section with a quick proof of Theorem~\ref{main thm}.
\begin{proof}[\textbf{Proof of Theorem~\ref{main thm}}]
Given constants $\mu>0$ and $t\in (0,1)$, we choose 
$$
\frac{1}{n}\ll \xi\ll\gamma\ll\mu.
$$
Let $G$ be an $n$-vertex weighted complete graph with $n\in 4\mathbb{N}$ and  $\delta^w(G)\ge (\frac{1+3t}{4}+\mu)n$. 
By Lemma~\ref{abs}, we can find a $(K_4,\xi)$-absorbing set $A\subseteq V(G)$ of size at most $\gamma n$. Let $G_1:=G-A$. Recall that $\gamma\ll\mu$. Hence 
\[
\delta^w(G_1)\ge\delta^w(G)-\gamma n\ge \left(\frac{1+3t}{4}+\mu\right)n-\gamma n\ge \left(\frac{1+3t}{4}+\frac{\mu}{2}\right)n.
\]
Applying Lemma~\ref{almost} on $G_1$, we obtain a $t$-heavy $K_4$-tiling $\mathcal{R}$ covering all but a set $R$ of at most $9+\frac{6}{\mu}$ vertices in $G_1$. Since $9+\frac{6}{\mu}\le\xi n$, the absorbing property of $A$ implies that $G[A\cup R]$ contains a $t$-heavy $K_4$-factor, which together with $\mathcal{R}$ forms a $t$-heavy $K_4$-factor in $G$.
\end{proof}

\subsection{Multicolored Regularity}

Let $G=(V,E,w)$ be a graph with an edge weighting $w: E\rightarrow [0,1]$. In order to better visualize the edge weights, we define an  edge coloring $c:E(G)\rightarrow [p]$ for any given integer $p\in \N$, where each edge $e\in E$ is assigned a color $\ell\in[p-1]$ if and only if $w(e)\in [\frac{\ell-1}{p},\frac{\ell}{p})$, and a color $p$ if and only if $w(e)\in [\frac{p-1}{p},1]$. For convenience, given an edge weighting $w$ and an integer $p$, we often use $c(w,p)$ to denote the edge coloring given as above, and simply call $G=(V,E,c)$ a \textit{$p$-edge-colored graph}. 
Note that $E(G)=\cup_{\ell \in [p]}E(G_{\ell})$ where $G_{\ell}$ is the weighted subgraph induced by all edges of color $\ell$.

Axenovich and Martin~\cite{AM} generalized the Szemer\'edi's regularity lemma to the multicolored version. 

\begin{defn}[Regular pair]
Given a $p$-edge-colored graph $G=(V,E, c)$. For any disjoint vertex subsets $X,Y\subseteq V$ and a color $\ell\in [p]$, the \textit{density} of the pair $(X,Y)$ in $G_\ell$ is defined as $$d_\ell(X,Y):=\frac{|E(G_{\ell}[X,Y])|}{|X||Y|}.$$ 
For a constant $\varepsilon>0$, the pair $(X,Y)$ is \textit{$\varepsilon$-regular} if for any $X'\subseteq X,\, Y'\subseteq Y$ with $|X'|\ge \varepsilon |X|,\,|Y'|\ge \varepsilon |Y|$, we have $$|d_\ell(X',Y')-d_\ell(X,Y)|<\varepsilon ~\text{for every}\ \ell\in [p].$$
Additionally, for a real vector ${\bf d}=(d_1,\ldots,d_p)$, if for  some $\ell\in [p]$, $d_\ell(X,Y)\ge d_{\ell}$, then we say $(X,Y)$ is \textit{$(\varepsilon,d_{\ell})$-regular} in $G_{\ell}$. Furthermore, we say $(X,Y)$ is \textit{$(\varepsilon,{\bf d})$-regular} in $G$ if $d_\ell(X,Y)\ge d_{\ell}$ for every $\ell\in [p]$.
\end{defn}


\begin{defn}[Regular partition]
For an $n$-vertex $p$-edge-colored graph $G=(V,E,c)$ and a constant $\varepsilon>0$, a partition $\mathcal{P}=\{V_0, V_1, \ldots , V_k\}$ is an \textit{$\varepsilon$-regular partition} of $V(G)$ if all of the following hold: 
\begin{itemize}
  \item $|V_0|\le \varepsilon n$;
  \item $|V_1|=|V_2|=\cdots=|V_k|$;
  \item all but at most $\varepsilon k^2$ pairs $(V_i,V_j)$ with $1\le i< j\le k$ are $\varepsilon$-regular.
\end{itemize}
We usually call $V_1,\ldots,V_k$ \textit{clusters} and call $V_0$ the \textit{exceptional set}.
\end{defn}

The following fact follows directly from the definition of regularity.  
\begin{fact}\label{fact: large deg}
Let $(X,Y)$ be an $(\varepsilon,{\bf d})$-regular pair in a $p$-edge-colored graph $G=(V,E,c)$. If $Y'\subseteq Y$ with $|Y'|\ge \varepsilon |Y|$, then all but at most $\varepsilon |X|$ vertices in $X$ have at most $(d_{\ell}+\varepsilon)|Y'|$ neighbors in $Y'$ for each color $\ell\in [p]$.
\end{fact}

Axenovich and Martin~\cite{AM} established a multicolored version of Szemer\'edi's regularity lemma, which can be deduced from the proof framework provided by Koml\'{o}s and Simonovits~\cite{KS}.

\begin{theorem}[Theorem 1.6 in~\cite{AM}]\label{thm: original form of RL}
Fix an integer $p\ge 2$. For every $\varepsilon>0$ and positive integer $K$, there exists an integer $M:=M(K,\varepsilon)$ with the following property:
for every $p$-edge-colored graph $G$ on $n\ge M$ vertices, there exists an $\varepsilon$-regular partition $\mathcal{P}=\{V_0,V_1,\dots,V_k\}$ of the vertex set $V(G)$, where $K\leq k\leq M$.
\end{theorem}


Following a standard `cleaning' step as outlined in \cite{Townsend2015}, we can derive a degree form of the  multicolored regularity lemma from Theorem~\ref{thm: original form of RL}. For completeness, we present its proof in Appendix A.


\begin{lemma}[Degree form of the multicolored regularity lemma]\label{thm: degree form of RL}
Fix an integer $p\ge 2$. For every $\eps>0$ there is an integer  $M:=M(\eps)$ such that the following holds for any real vector ${\bf d}=(d_1,\ldots,d_p)\in [0,1]^p$ and every integer $n\geq M$. Let $G$ be a $p$-edge-colored graph on $n$ vertices. Then there exists an $\eps$-regular partition $\mathcal{P}=\{V_0,V_1, \dots,V_k\}$ of $V(G)$ and a spanning subgraph $G'\subseteq G$ with the following properties:
\begin{enumerate}[label =\rm (\arabic{enumi})]
 \item\label{pro1} $\frac{1}{\eps}\le k \le M$;
 \item\label{pro2} $|V_0|\le \eps n$ and $|V_1|=|V_2|=\dots=|V_k|=:m$; 
 \item\label{pro3} $d_{G'_{\ell}}(v)>d_{G_\ell}(v)-(d_{\ell}+\eps)n$ for every $v\in V(G)$ and $\ell \in [p]$;
 \item\label{pro4} $|E(G_\ell'[V_i])|=0$ for every $i\in [k]$ and $\ell \in [p]$;
 \item\label{pro5} for $1\le i<j\le k$ and $\ell\in[p]$, each pair $(V_i,V_j)$ is $\eps$-regular with density either $0$ or at least $d_{\ell}$ in $G'_{\ell}$.
\end{enumerate}
\end{lemma}


A widely-used auxiliary graph accompanied with the regular partition is the \textit{reduced graph} defined as follows. 
\begin{defn}[Reduced graph]
Let $G$ be a $p$-edge-colored weighted graph  with the corresponding edge coloring $c(w,p)$. Given parameter $\eps>0$, let $G'$ be a spanning subgraph of $G$ and $\mathcal{P}=\{V_0,V_1, \dots,V_k\}$ be a partition of $V(G)$, as given in Lemma \ref{thm: degree form of RL}. We define the \textit{edge-weighted reduced graph} $R:=R(\eps)$ for $\mathcal{P}$ as follows: 
$R$ is a complete graph on vertex set $\{V_1,\ldots,V_k\}$ and the weight of every edge $V_iV_j$ is defined as 
\[w_{R}(V_iV_j):=\sum_{\ell\in [p]}\frac{\ell-1}{p}\cdot d_\ell(V_i,V_j),\]
where $d_\ell(V_i,V_j)$ denotes the density of the pair $(V_i,V_j)$ in the subgraph $G'_{\ell}$.

\end{defn}


The following lemma tells us that the weighted reduced graph inherits the minimum weighted degree from the original graph. 

\begin{lemma}\label{lem: degree of reduced graph}
For a non-negative vector ${\bf d}=(d_1,\ldots,d_p)$ and positive constants $\mu, \varepsilon, c$, 
let $G$ be an $n$-vertex $p$-edge-colored weighted graph  with the corresponding edge coloring $c(w,p)$ and  $\delta^w(G)\ge (c+\mu)n$. Assume that  $G'$ and $\mathcal{P}$ are obtained from Lemma~\ref{thm: degree form of RL}, and $R:=R(\eps)$ is the corresponding edge-weigthed reduced graph. Then $\delta^w(R)\ge (c+\mu-\sum_{\ell\in [p]}d_{\ell}-2p\varepsilon-\tfrac{1}{p} )k$.
\end{lemma}

\begin{proof}
By definition, each $V_i\in V(R)$ has a weighted degree \allowdisplaybreaks
\begin{align}\notag
d^w(V_i)&=\sum_{j\neq i}\sum_{\ell\in [p]}\frac{\ell-1}{p}\cdot d_\ell(V_i,V_j)=\sum_{j\neq i}\sum_{\ell\in [p]}\frac{\ell-1}{p}\cdot \frac{|E(G_{\ell}'[V_i,V_j])|}{|V_i||V_j|}\\\notag
&=\frac{1}{m^2}\sum_{j\neq i}\sum_{\ell\in [p]}\frac{\ell-1}{p}\sum_{v\in V_i}|N_{G'_{\ell}}(v)\cap V_j|\\\notag
&= \frac{1}{m^2}\sum_{v\in V_i}\sum_{\ell\in [p]}\frac{\ell-1}{p}\cdot  (d_{G'_{\ell}}(v)-|V_0|)\ \ \  \text{(by Lemma \ref{thm: degree form of RL} \ref{pro4}-\ref{pro5})}\\\notag
&\geq \frac{1}{m^2}\sum_{v\in V_i}\sum_{\ell\in [p]}\frac{\ell-1}{p}\cdot \big{(}d_{G_{\ell}}(v)-(d_{\ell}
+2\eps) n\big{)}\ \ \ \text{(by Lemma \ref{thm: degree form of RL} \ref{pro3})}\\\label{eq:weight}
&\geq \frac{1}{m^2}\sum_{v\in V_i}\sum_{\ell\in [p]} \big{(}d^w_{G_{\ell}}(v)-\tfrac{\ell}{p}\cdot (d_{\ell}+2\eps)n-\frac{d_{G_{\ell}}(v)}{p}\big{)}\\\notag
&\geq \frac{1}{m^2}\sum_{v\in V_i}\big{(}\delta^w(G)-\frac{n}{p}-\sum_{\ell\in [p]}\tfrac{\ell}{p}\cdot(d_{\ell}+2\eps)n\big{)}\\\notag
&\geq\big{(}c+\mu-\sum_{\ell\in [p]}d_{\ell}-2p\eps-\tfrac{1}{p}\big{)}k,
\end{align}
where the inequality in \eqref{eq:weight} follows since the weight of each edge in graph $G_{\ell}$ is at most $\frac{\ell}{p}$. Consequently, $\delta^w(R)\ge (c+\mu-\sum_{\ell\in [p]}d_{\ell}-2p\varepsilon-\tfrac{1}{p} )k$. 
\end{proof}

At the end of this subsection, we establish an embedding lemma that will be instrumental in our proof. We show that every heavy $K_4$ in the reduced weighted graph $R$ ensures at least $\Omega(n)$ vertex-disjoint heavy $K_4$s in the original weighted  graph $G$.

\begin{lemma}[Embedding Lemma]\label{lem: embedding lem}
For all $d>0,\,t\in(0,1)$ and $\frac{1}{n}\ll \be\ll \eps\ll d$, assume that $G$ is a $p$-edge-colored weighted complete graph on $n$ vertices and $\mathcal{P}=\{V_0,V_1,\ldots,V_k\}$ is the partition of $V(G)$ obtained in Lemma~\ref{thm: degree form of RL}. Let ${\bf d}=(d,\ldots,d)$  and $R:=R(\eps)$ be the weighted reduced graph for $\mathcal{P}$. If $R$ contains a $t$-heavy $K_4$ induced on $\{V_1,\cdots,V_4\}$, then $G$ contains $\be n$ vertex-disjoint $t$-heavy $K_4$'s each containing exactly one vertex in $V_i$ for $i\in[4]$. 
\end{lemma}
\begin{proof}
Let $w_R$ be the edge-weighting of $R$ with $w_{ij}:=w_R(V_iV_j)=\sum_{\ell\in [p]}\frac{\ell-1}{p}\cdot d_{\ell}(V_i,V_j)$ for each $V_iV_j\in E(R)$. Since $R[\{V_1,\ldots,V_4\}]$ is heavy, one has $\sum_{1\le i<j\le 4}w_{ij}>6t$.
For any $1\leq i<j\leq 4$ with $w_{ij}>0$, it follows by definition that there exists $\ell_{ij}\in [p]$ such that  $(V_i,V_j)$ is $(\eps, d)$-regular in the subgraph $G_{\ell_{ij}}'$ and $\frac{\ell_{ij}-1}{p}\ge w_{ij}$. Otherwise, we have $\frac{\ell-1}{p}< w_{ij}$ for all $\ell\in [p]$, and this yields a contradiction as 
\[
w_{ij}=\sum_{\ell\in [p]}\frac{\ell-1}{p}\cdot d_{\ell}(V_i,V_j)<w_{ij}\cdot\sum_{\ell\in [p]}d_{\ell}(V_i,V_j)\le w_{ij}.
\]
If $w_{ij}=0$, then we arbitrarily pick an integer $\ell_{ij}\in[p]$.

Now we have a collection of $\eps$-regular bipartite subgraphs $\{G'_{\ell_{ij}}[V_i,V_j]:1\le i<j\le 4\}$ such that for each $w_{ij}>0$, $(V_i,V_j)$ is $(\eps, d)$-regular in $G'_{\ell_{ij}}$ and $\frac{\ell_{ij}-1}{p}\ge w_{ij}$. By applying a standard embedding lemma and by the choice $\be\ll \eps\ll d$, we can greedily find $\be n$ vertex-disjoint copies of $K_4$ as desired in $G$, and each of them has  total edge weight at least \[\sum_{1\le i<j\le 4,~w_{ij}>0}\tfrac{\ell_{ij}-1}{p}\ge \sum_{1\le i<j\le 4}w_{ij}>6t.\] This completes the proof.
\end{proof}

\subsection{Geometric perspectives on heavy faces in Tetrahedrons}
In this section, we collect a series of interesting observations that will be useful in our proof. 
\begin{prop}\label{prop: heavy}
The following results hold for any fixed weighted complete graph $G$ and $t\in(0,1)$.
\begin{enumerate}[label =\rm  (\arabic{enumi})]
 \item\label{per1}  Let $T$ be a heavy triangle in $G$ and $u$ be a vertex in $V(G)\setminus V(T)$ such that $w(u,T)>3t$. Then $G[V(T)\cup \{u\}]$ forms a heavy $K_4$ containing at least two heavy triangles and two heavy edges.
 \item\label{per2}  Let $uu'$ and $vv'$ be two disjoint heavy edges of $G$ such that $w(uu',vv')>4t$. Then $G[\{u,u',v,v'\}]$ is a heavy $K_4$ containing  at least two heavy triangles and two heavy edges.
    \end{enumerate}
\end{prop}
\begin{proof}
(1) Let $G_1:=G[V(T)\cup \{u\}]$ and $V(T):=\{u_1,u_2,u_3\}$. It is easy to see that $G_1$ contains at least two heavy edges. Suppose to the contrary that $u_1u_2u, u_1u_3u$ and $u_2u_3u$ are not heavy triangles.
Hence 
\[
3t+6t< w(u_1u_2u_3)+2w(u_1u_2u_3,u)= w(u_1u_2u)+w(u_1u_3u)+w(u_2u_3u)\le 9t,
\]
a contradiction.
It follows that $G_1$ contains a heavy triangle different from $T$, as desired. 

(2) Clearly, $G[\{u,u',v,v'\}]$ is a heavy $K_4$ containing at least two heavy edges. Since $w(uu',vv')>4t$, either 
$w(u,vv')>2t$ or $w(u',vv')>2t$ holds. Hence, at least one of $uvv'$ and $u'vv'$ is a heavy triangle in $G$. Similarly, either 
$vuu'$ or $v'uu'$ is a heavy triangle in $G$, as desired. 
\end{proof}

\noindent\textbf{Remark.} It is worth mentioning that Proposition~\ref{prop: heavy}~\ref{per1} can be extended for larger cliques $K_r$ with $r\ge 4$:  Let $T$ be a heavy $K_r$ in $G$ and $u$ be a vertex such that $w(u,T)>rt$. Then $G[V(T)\cup \{u\}]$ forms a heavy $K_{r+1}$ containing at least two heavy $K_r$s and two heavy edges. The proof is almost the same as above and here we omit it.

\begin{prop}\label{prop: heavy triangle}
For all $t\in(0,1)$ and $\mu>0$, if $G$ is a weighted complete graph with $\delta^w(G)\geq (\frac{1+3t}{4}+\mu)n$, then the following results hold for sufficiently large $n$.
\begin{enumerate}[label =\rm  (\arabic{enumi})]
 \item\label{quarter1} If $G$ contains a heavy $K_{r-1}$ for some integer $r\geq 2$, then there are at least $\frac{1}{4}n$ vertices $u\in V(G)\setminus V(K_{r-1})$ such that $w(V(K_{r-1}),u)> (r-1)t$.
Furthermore, $V(K_{r-1})\cup\{u\}$ induces a heavy $K_r$ in $G$.
 \item\label{quarter2} If $S\<V(G)$ is a set of $4$ vertices, then there are at least $\mu n$ vertices $u\in V(G)\setminus S$ such that $w(S,u)>1+3t$.
\end{enumerate}
\end{prop}
\begin{proof}
We only give the proof of \ref{quarter1}, and the second part can be proved by a similar argument. Fix a heavy $K_{r-1}$ with vertex set $S$ in $G$, let $V':=V(G)\setminus S$ and 
\[
X:=\{u\in V': w(u,S)> (r-1)t\}.
\]
Since $\delta^w(G)\ge \left(\frac{1+3t}{4}+\mu\right)n$,  one has 
\begin{align*}
(r-1)\cdot \left(\frac{1+3t}{4}+\mu\right)n-2\cdot \binom{r-1}{2}\le w(S,V')\le (r-1)\cdot |X|+(r-1)t\cdot |V'\setminus X|.
\end{align*}
This implies that $|X|\ge \frac{1}{4}n$, as desired. 
\end{proof}

\section{Almost cover}
In this section, we prove the almost cover lemma by using our new perspectives (in Proposition~\ref{prop: heavy}) on heavy tetrahedrons. 
\begin{proof}[\textbf{Proof of Lemma~\ref{almost}}] 
Given {constants} $\mu>0$ and $t\in(0,1)$, we assume that $G$ is an $n$-vertex weighted complete graph with $\delta^w(G)\ge (\frac{1+3t}{4}+\mu)n$. 

Let $\mathcal{R}$ be a collection of vertex-disjoint heavy $K_4$s in $G$, each of which contains at least two heavy triangles and two heavy edges. Define $\rho:=\sum_{R\in \mathcal{R}}w(R)$. Let $\mathcal{T}$ be a collection of vertex-disjoint heavy triangles in $G[V\setminus V(\mathcal{R})]$, let $M$ be a matching consisting of  heavy edges in $G[V\setminus V(\mathcal{R}\cup \mathcal{T})]$, and denote $I:=V\setminus V(\mathcal{R}\cup \mathcal{T}\cup M)$. Assume that $\mathcal{R}$,  $\mathcal{T}$, and $M$ are picked to maximize the tuple $(|\mathcal{R}|,  |\mathcal{T}|,|M|,\rho)$ lexicographically. 
In the subsequent three claims, we show that none of $|\mathcal{T}|$, $|M|$, or $|I|$ is large. 
\begin{claim}\label{claim-T-size}
    $|\mathcal{T}|\leq 3$.
\end{claim}
\begin{proof}
Suppose for contradiction that there are at least $4$ distinct heavy triangles in $\mathcal{T}$, say $T_1,T_2,T_3$ and $T_4$. Denote by $S$ the vertex set of $\cup_{i\in [4]}T_i$. By the maximality of $|\mathcal{R}|$ and Proposition~\ref{prop: heavy}~\ref{per1}, for each $i\in [4]$ and any vertex $x\in V\setminus V(\mathcal{R}\cup T_i)$, we have $w(T_i,x)\leq 3t$. Therefore, 
    \begin{align*}
        w(S, V(\mathcal{R}))
        &\geq12\cdot \left(\frac{1+3t}{4}+\mu\right)n- 
        4\cdot 3t\cdot |V\setminus V(\mathcal{R})|-4\cdot 6\\
        &\geq 3(1+3t+\mu)|V(\mathcal{R})|\\
        &=12(1+3t+\mu)|\mathcal{R}|.
    \end{align*}
Thus, there exists a heavy $K_4$ say $R\in\mathcal{R}$ such that $w(S,V(R))\geq 12(1+3t+\mu)$. Without loss of generality, assume
that $w(T_1,R)\geq w(T_2,R)\geq w(T_3,R)\geq w(T_4,R)$. Hence $w(T_1,R)> 3(1+3t)$ and there are two vertices $v_1,v_2\in V(R)$ such that $w(T_1,v_1)\geq w(T_1,v_2)>3t$. Observe that $w(T_1,R)\leq 12$. Then $w(T_2,R)>12t.$  It follows that there is a vertex $v'\in V(R)$ such that $w(T_2,v')> 3t$. Choose $v\in \{v_1,v_2\}\setminus \{v'\}$. Together with Proposition~\ref{prop: heavy}~\ref{per1}, $G[\{v\}\cup V(T_1)]$ and $G[\{v'\}\cup V(T_2)]$ are two heavy $K_4$s, each containing at least two heavy triangles and two heavy edges (see Figure \ref{fig1}), which violates the maximality of $|\mathcal{R}|$.
\begin{figure}[h]
    \centering
\begin{tikzpicture}[
    dot/.style={circle,fill=black,inner sep=1.5pt},
    highlightdot/.style={
        circle,
        fill=red!80,
        draw=red,
        line width=1.2pt,
        inner sep=1.5pt
    },
    shaded triangle/.style={
        fill=red!50,
        draw=black,
        line width=0.6pt
    },
    highlight line 1/.style={  
        draw=red,           
        line width=1.5pt    
    }
]
\coordinate (A1) at (1,2);
\coordinate (A2) at (1.75,2.5);
\coordinate (A3) at (2.5,2);
\coordinate (A4) at (1.75,3);

\draw (A1) -- (A2)-- (A3) -- (A4)-- (A1)-- (A3);
\draw (A2) -- (A4);

\node[dot] at (A1) {};
\node[dot] at (A2) {};
\node[dot] at (A3) {};
\node[dot] at (A4) {};

\coordinate (A) at (-0.5,0.5);
\coordinate (B) at (0.5,0.5);
\coordinate (C) at (1.5,0.5);


\draw[blue, line width=1pt] (A) -- (B) -- (C);
\draw[blue, line width=1pt] (A) -- (C);
\draw[blue, line width=1pt] (A) to[out=-30,in=-150] (C); 

\coordinate (B1) at (2,0.5);
\coordinate (B2) at (3,0.5);
\coordinate (B3) at (4,0.5);


\draw[red, line width=1pt] (B1) -- (B2) -- (B3);
\draw[red, line width=1pt] (B1) -- (B3);
\draw[red, line width=1pt] (B1) to[out=-30,in=-150] (B3); 

\node[dot] at (A) {};
\node[dot] at (B) {};
\node[dot] at (C) {};
\node[dot] at (B1) {};
\node[dot] at (B2) {};
\node[dot] at (B3) {};

\draw[blue, line width=1pt] (A1) -- (A);
\draw[blue, line width=1pt] (A1) -- (B);
\draw[blue, line width=1pt] (A1) -- (C);
\draw[red, line width=1pt] (A3) -- (B1);
\draw[red, line width=1pt] (A3) -- (B2);
\draw[red, line width=1pt] (A3) -- (B3);

\node at (0.7,2) {$v$};
\node at (2.9,2) {$v'$};
\node at (0.5,-0.1) {$T_1$};
\node at (3,-0.1) {$T_2$};

\end{tikzpicture}
    \caption{Claim \ref{claim-T-size}.}
    \label{fig1}
\end{figure}
\end{proof}

\begin{claim}\label{claim-M-size}
    $|M|\leq \frac{1}{\mu}$.
\end{claim}
\begin{proof}
Suppose that $M=\{e_1,e_2,\ldots,e_{K}\}$ with $e_i=u_iu_i'$ and $K>\frac{1}{\mu}$. By a similar argument as Claim \ref{claim-T-size}, we obtain that there exists a heavy $K_4$, say $R\in\mathcal{R}$ with $V(R)=\{v_1,v_2,v_3,v_4\}$,  satisfying $w(\cup_{i=1}^Ke_i,R)\geq 2K(1+3t+3\mu)$.

Without loss of generality, assume that $w(e_i,R)\geq w(e_j,R)$  and $w(u_i,R)\geq w(u_i',R)$ for each $1\leq i<j\leq K$. Hence $w(e_1,R)> 2(1+3t)$ and 
$$w(e_2,R)\geq \frac{2K(1+3t+3\mu)-8}{K-1}>2(1+3t).$$
Therefore, $w(u_1,R)>  1+3t$ and $w(u_2,R)>1+3t$. 
     
Recall that $R$ contains at least two heavy triangles. Without loss of generality, assume that $v_1v_2v_3$ and $v_1v_3v_4$ are two such triangles. If $w(v_2,e_1)>2t$, then $v_2u_1u_1'$ is a heavy triangle inside $G$. On the other hand, it follows from Proposition~\ref{prop: heavy}~\ref{per1} that $u_2v_1v_3v_4$ is a heavy $K_4$ containing at least two heavy triangles and two heavy edges (see Figure \ref{fig2}). This contradicts the maximality of $|\mathcal{T}|$. Hence  $w(v_2,e_1)\leq 2t$. Similarly, we have $w(v_2,e_2),w(v_4,e_1),w(v_4,e_2)\leq 2t$. In this case, for every $i\in [2]$ we have 
\begin{align*}   &w(v_1v_2v_3,e_i)> 2+4t, \  w(v_1v_3v_4,e_i)> 2+4t,\\
&w(v_1v_3,e_i)> 2+2t,\ w(v_2v_3,e_i)> 4t\ \text{and}\\
&w(v_1,e_i),\,w(v_3,e_i)>2t.
\end{align*}

If $v_2v_3$ is a heavy edge, then Proposition~\ref{prop: heavy}~\ref{per2} implies that $v_2v_3u_1u_1'$ is a heavy $K_4$ containing at least two heavy triangles and two heavy edges. Furthermore, $v_1u_2u_2'$ forms a heavy triangle (see Figure \ref{fig3}), which also  contradicts the maximality of $|\mathcal{T}|$. It follows that $v_iv_{i+1}$ is not a heavy edge for any $i\in [4]$ (where $v_5=v_1$). Recall that  $R$ contains at least two heavy edges. Hence $w(v_1v_3),\, w(v_2v_4)>t$.

\begin{figure}[h]
\begin{minipage}[t]{0.48\textwidth}
    \centering
\begin{tikzpicture}[
    dot/.style={circle,fill=black,inner sep=1.5pt},
    shaded triangle/.style={
        fill=red!50,  
        draw=black,    
        line width=0.6pt
    }
]

\coordinate (A1) at (0,2);
\coordinate (A2) at (-0.5,1);
\coordinate (A3) at (0,0);
\coordinate (A4) at (0.5,1);

\draw (A1) -- (A2)-- (A4)-- (A1);
\draw (A2) -- (A4);
\draw (A2) -- (A3);
\draw[red, line width=1pt] (A1) -- (A3) -- (A4)-- (A1);
\node[dot] at (A1) {};
\node[dot] at (A2) {};
\node[dot] at (A3) {};
\node[dot] at (A4) {};

\coordinate (A) at (-1,1.5);
\coordinate (B) at (-1,0.5);

\draw[blue, line width=1pt] (A) -- (B); 

\coordinate (B1) at (1.2,1.5);
\coordinate (B2) at (1.2,0.5);

\draw (B1) -- (B2); 

\node[dot] at (A) {};
\node[dot] at (B) {};
\node[dot] at (B1) {};
\node[dot] at (B2) {};

\draw[blue, line width=1pt] (A2) -- (A);
\draw[blue, line width=1pt] (A2) -- (B);
\draw[red, line width=1pt] (B1) -- (A1);
\draw[red, line width=1pt] (B1) -- (A3);
\draw[red, line width=1pt] (B1) -- (A4);

\node at (-1.2,1) {$e_1$};
\node at (1.4,1) {$e_2$};
\node at (0,2.2) {$v_1$};
\node at (-0.6,1.3) {$v_2$};
\node at (0,-0.2) {$v_3$};
\node at (0.6,1.3) {$v_4$};
\node at (1.2,1.7) {$u_2$};

\end{tikzpicture}
    \caption{$w(v_2,e_1)>2t.$}
    \label{fig2}
\end{minipage}
\hfill
\begin{minipage}[t]{0.48\textwidth}
\centering
    \begin{tikzpicture}[
    dot/.style={circle,fill=black,inner sep=1.5pt},
    shaded triangle/.style={
        fill=red!50,  
        draw=black,    
        line width=0.6pt
    }
]

\coordinate (A1) at (0,2);
\coordinate (A2) at (-0.5,1);
\coordinate (A3) at (0,0);
\coordinate (A4) at (0.5,1);

\draw (A1) -- (A2);
\draw(A3) -- (A4)-- (A1)-- (A3);
\draw (A2) -- (A4);
\draw[red, line width=1pt](A2) -- (A3);

\node[dot] at (A1) {};
\node[dot] at (A2) {};
\node[dot] at (A3) {};
\node[dot] at (A4) {};

\coordinate (A) at (-1,1);
\coordinate (B) at (-1,0);

\draw[red, line width=1pt] (A) -- (B); 

\coordinate (B1) at (1.2,1.5);
\coordinate (B2) at (1.2,0.5);

\draw[blue, line width=1pt] (B1) -- (B2); 

\node[dot] at (A) {};
\node[dot] at (B) {};
\node[dot] at (B1) {};
\node[dot] at (B2) {};

\draw[red, line width=1pt] (A2) -- (A);
\draw[red, line width=1pt] (A2) -- (B);
\draw[red, line width=1pt] (A3) -- (A);
\draw[red, line width=1pt] (A3) -- (B);
\draw[blue, line width=1pt] (B1) -- (A1);
\draw[blue, line width=1pt] (B2) -- (A1);

\node at (-1.2,0.5) {$e_1$};
\node at (1.4,1) {$e_2$};
\node at (0,2.2) {$v_1$};
\node at (-0.6,1.3) {$v_2$};
\node at (0,-0.2) {$v_3$};
\node at (0.6,0.7) {$v_4$};

\end{tikzpicture}
\caption{$v_2v_3$ is a heavy edge.}
    \label{fig3}
\end{minipage}
\end{figure}
Applying Proposition~\ref{prop: heavy}~\ref{per2} again yields that $u_1u_1'v_1v_3$ is a heavy $K_4$ containing at least two heavy triangles and two heavy edges.  Since $v_2v_4$ is a heavy edge, by the maximality of $\rho$ one has $w(u_1u_1'v_1v_3)\leq w(R)$, that is, 
$$w(v_1v_3)+2+2t+t< w(u_1u_1'v_1v_3)\leq w(R)\leq w(v_1v_3)+4t+1,
$$
 and this implies $t>1$, a contradiction. 
\end{proof}
\begin{claim}\label{claim-I-size}
    $|I|\leq \frac{1}{\mu}$. 
\end{claim}
\begin{proof}
Suppose that $I=\{u_1,u_2,\ldots,u_{K}\}$ with $K>\frac{1}{\mu}$. By a similar averaging argument as in Claim \ref{claim-T-size}, we obtain that there exists a heavy $K_4$ in $\mathcal{R}$, say $R$ with $V(R)=\{v_1,v_2,v_3,v_4\}$, such that $w(\cup_{i=1}^Ku_i,R)\geq K(1+3t+3\mu)$. We may assume that $w(u_i,R)\geq w(u_j,R)$ for all $1\leq i<j\leq K$. Hence $w(u_1,R)>1+3t$ and then $$ w(u_2,R)\geq \frac{K(1+3t+3\mu)-4}{K-1}>1+3t.$$
Furthermore, we may assume that $v_1v_2v_3$ and $v_1v_3v_4$ are two heavy triangles in $R$. Then  Proposition~\ref{prop: heavy}~\ref{per1} implies that both $u_2v_1v_2v_3$ and $u_2v_1v_3v_4$ are heavy $K_4$s in $G$, each of which contains at least two heavy triangles and two heavy edges.
By the maximality of $|M|$, we obtain $w(v_4u_1),w(v_2u_1)\leq t$. Similarly, none of $u_2v_2,u_2v_4$ is a heavy edge. Therefore, 
$$
     w(u_1v_1),\,w(u_1v_3),\,w(u_2v_1),\,w(u_2v_3)>t.
$$

     
On the other hand, the maximality of $\rho$ guarantees that 
     $$
     w(v_4,v_1v_2v_3)\geq w(u_2,v_1v_2v_3)>3t\ \text{and}\  w(v_2,v_1v_3v_4)\geq w(u_2,v_1v_3v_4)>3t.
     $$
     Therefore,
     $$
     6t<w(v_4,v_1v_2v_3)+w(v_2,v_1v_3v_4)=w(v_1v_2v_4)+w(v_2v_3v_4).
     $$
     It follows that either $v_1v_2v_4$ or $v_2v_3v_4$ is a heavy triangle in $G$. We only consider the former case. 
     Applying Proposition~\ref{prop: heavy}~\ref{per1} again yields that $u_1v_1v_2v_4$ is a heavy 
$K_4$ containing at least two heavy triangles and two heavy edges. However, as $u_2v_3$ is a heavy edge (see Figure \ref{fig4}), this violates the maximality of $|M|$. 
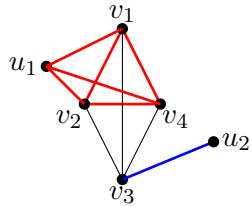
\begin{figure}[h]
    \centering
\begin{tikzpicture}[
    dot/.style={circle,fill=black,inner sep=1.5pt},
    shaded triangle/.style={
        fill=red!50,  
        draw=black,    
        line width=0.6pt
    }
]

\coordinate (A1) at (0,2);
\coordinate (A2) at (-0.5,1);
\coordinate (A3) at (0,0);
\coordinate (A4) at (0.5,1);

\draw[red, line width=1pt] (A1) -- (A2)-- (A4)-- (A1);
\draw (A2) -- (A3);
\draw (A1) -- (A3) -- (A4);
\node[dot] at (A1) {};
\node[dot] at (A2) {};
\node[dot] at (A3) {};
\node[dot] at (A4) {};

\coordinate (A) at (-1,1.5);

\coordinate (B2) at (1.2,0.5);

\draw[blue, line width=1pt] (A3) -- (B2); 

\node[dot] at (A) {};
\node[dot] at (B2) {};


\draw[red, line width=1pt] (A) -- (A1);
\draw[red, line width=1pt] (A) -- (A2);
\draw[red, line width=1pt] (A) -- (A4);

\node at (-1.3,1.5) {$u_1$};
\node at (0,2.2) {$v_1$};
\node at (-0.7,0.8) {$v_2$};
\node at (0,-0.2) {$v_3$};
\node at (0.7,0.8) {$v_4$};
\node at (1.5,0.5) {$u_2$};

\end{tikzpicture}
    \caption{$v_1v_2v_4$ is a heavy triangle.}
    \label{fig4}
    \end{figure}
\end{proof}

Combining Claims \ref{claim-T-size}-\ref{claim-I-size}, we know that at most $9+\frac{3}{\mu}$ vertices are not covered by $\mathcal{R}$. This completes the proof. 
\end{proof}

\section{Building an absorbing set}
In this section, we discuss how to construct absorbing sets in weighted graphs using the lattice-based absorbing method. 
Here we adopt an approach due to Nenadov and Pehova~\cite{Nenadov2020}, which roughly tells that an absorbing set is guaranteed whenever every $4$-set of vertices 
witnesses linearly many vertex-disjoint absorbers. We use their approach in the edge-weighted setting and obtain the following result concerning heavy $K_4$, where the proof is by now standard and thus deferred to Appendix B of the arXiv version.

\begin{lemma}\label{lem: abs set}
For any $\gamma>0,\,t\in(0,1)$ and $s\in\N$, there exist $\xi>0$ and $n_0\in \mathbb{N}$ such that the following holds. If $G$ is a weighted  graph on $n\geq n_0$ vertices such that  every set of four vertices has at least $\gamma n$ vertex-disjoint $(K_4,s)$-absorbers, then $G$ contains a $(K_4, \xi)$-absorbing set $A$ of size at most $\gamma n$.
\end{lemma}
This leaves the difficult task of constructing linearly many vertex-disjoint absorbers for every four vertices, which we now discuss.

\subsection{Constructing absorbers}
We first introduce a crucial notion of reachablility which was first introduced by Lo and Markstr\"om~\cite{Lo2015} and here we use a variant from~\cite{Chang2023}.
\begin{defn}[Reachability \& Closedness]    
Let $m,s\in \mathbb{N},\, t\in(0,1)$ and $G$ be an $n$-vertex weighted graph. Suppose $H$ is an $h$-vertex graph.
We say that two vertices $u,v\in V(G)$ are \textit{$(H,m,s)$-reachable} if for any set $W$ of $m$ vertices, there is a set $S\subseteq V(G)\setminus W$ of size at most $hs-1$ such that both $G[S\cup\{u\}]$ and $G[S\cup\{v\}]$ contain $t$-heavy $H$-factors. We say a set $U\subseteq V(G)$ is \textit{$(H, m, s)$-closed} if every pair of vertices $u,v\in U$ is $(H, m, s)$-reachable in $G$; and $U$ is \textit{$(H, m, s)$-inner closed} if every pair of vertices $u,v\in U$ is $(H, m, s)$-reachable in the subgraph $G[U]$.
\end{defn}

The following result shows that for every $4$-vertex set $S$, one can construct linearly many vertex-disjoint absorbers, provided that $V(G)$ is closed.
This can be obtained by the following argument in~\cite{Han2024} in the edge-weighted setting, and we postpone the proof in Appendix C of the arXiv version.

\begin{lemma}\label{lem: absorber lem}
For all $t\in (0,1)$, $\beta>0$ and $s\in\N$, the following holds for sufficiently large $n\in \mathbb{N}$. Let $G$ be an $n$-vertex weighted  graph such that $V(G)$ is $(K_4,  \beta n,s)$-closed. Then every $S\in {V(G)\choose 4}$ has at least $\frac{\beta}{64s}n$ vertex-disjoint $(K_4,s)$-absorbers.
\end{lemma}

Based on Lemma~\ref{lem: abs set} and Lemma~\ref{lem: absorber lem}, a key step of constructing linearly many vertex-disjoint absorbers is to show that $V(G)$ is a closed set. Here we are able to replace this property with a slightly weaker one in the following, which is the bulk of our task. 
We present its proof to the next section.

\begin{lemma}\label{main lemma}
For all $\mu>0$ and $t\in (0,1)$, there exist $\gamma_0, \beta>0$ and an integer $n_0\in \mathbb{N}$ such that the following holds for all $n\geq n_0$ and $\gamma\le\gamma_0$. Let $G$ be an $n$-vertex weighted  complete graph with $\de^w(G)\ge (\frac{1+3t}{4}+\mu)n$. Then $V(G)$ admits a partition $V(G)=B\cup U$ such that $|B|\le \gamma n$ and $U$ is $(K_4,\beta n,32)$-inner closed.
\end{lemma}

In this case, we apply Lemma~\ref{lem: abs set} and Lemma~\ref{lem: absorber lem} on $G[U]$ to obtain an absorbing set in $U$. It remains to handle vertices in $B$. For this, we use a result of Balogh, Kemkes, Lee and Young~\cite{BKLY}, which ensures that every vertex is covered by $\Omega(n^3)$ copies of heavy $K_4$. 

\begin{lemma}[\cite{BKLY}]\label{lem: cover threshold}
For every $t\in (0,1)$ and $r\geq 3$, if $G$ is an $n$-vertex weighted complete graph with minimum weighted degree at least $\de n$, then every vertex is in at least $\left(1-\frac{1-\de}{1-t}\right)\binom{n-1}{r-1}$ $t$-heavy $K_r$s.
\end{lemma}

Now, we give a quick proof of Lemma~\ref{abs}. 

\begin{proof}[\textbf {Proof of Lemma~\ref{abs}}]
Given $t\in(0,1)$ and $\mu>0$, we choose $\tfrac{1}{n}\ll \xi\ll \beta\ll \gamma\ll \mu$ 
and let $G$ be an  $n$-vertex weighted complete graph  with $\de^w(G)\ge (\frac{1+3t}{4}+\mu)n$.
By Lemma \ref{main lemma}, 
$V(G)$ admits a partition $\mathcal{P}=\{B,U\}$ such that $|B|\le \frac{\gamma}{9} n$ and $U$ is $(K_4,\beta n,32)$-inner closed. Lemma \ref{lem: absorber lem} implies that every $S\in {U\choose 4}$ has at least $\frac{\beta}{2^{11}}n$ vertex-disjoint $(K_4,32)$-absorbers. Together with Lemma \ref{lem: abs set}, we obtain that 
$G[U]$ contains a $(K_4,  \xi)$-absorbing set $A$ of size at most $\frac{\beta}{2^{11}} n$. 

On the other hand, by Lemma~\ref{lem: cover threshold}, we know that every vertex in $G$ lies in at least $\frac{1}{4}\binom{n-1}{3}$  heavy $K_4$s. Hence, we can greedily choose $|B|$ vertex-disjoint heavy $K_4$s in $G-A$ each of which containing exactly one vertex in $B$ as $|B|\le \frac{\gamma}{9} n$. Let $F$ be the vertex set of those vertex-disjoint heavy $K_4$s. Since $\beta\ll \gamma$, the set $A\cup 
F$ forms a $(K_4,  \xi)$-absorbing set of size at most $\gamma n$ in $G$, as desired.
\end{proof}

\subsection{Reachability and Latticed-based Absorption}

The proof of Lemma \ref{main lemma} is technical; a high-level outline is as follows.
\begin{enumerate}
    \item Show that for every three vertices, at least two of them are $1$-reachable (see Lemma~\ref{lem: two reachable from three}). 
    \item Prove that $V(G)$ admits a partition $\mathcal{P}$ with at most $2$ parts, each of which is closed (see Lemma~\ref{lem: partition lem}).
    \item The key step is 
the merging process, where we first apply the multicolor version regularity lemma, then employ transferrals in the reduced graph  (see Lemma~\ref{lem: merge lem}) and finally invoke an embedding lemma (see Lemma~\ref{lem: embedding lem}). 
\end{enumerate}

\begin{lemma}\label{lem: two reachable from three}
Given $t\in (0,1)$, $\beta,\mu>0$ with $\be< \frac{\mu}{2}$, there exists an integer 
$n_0\in \mathbb{N}$ such that the following holds for every $n\ge n_0$. Let $G$ be an $n$-vertex weighted  complete graph with $\de^w(G)\ge (\frac{1+3t}{4}+\mu)n$. Then among every three vertices of $V(G)$, there exist two vertices that are $(K_4,\beta n,1)$-reachable.
\end{lemma}
\begin{proof}
Given constants $t\in(0,1),\,\beta,\mu>0$ with $\beta<\frac{\mu}{2}$, we choose $\frac{1}{n}\ll \beta,\mu$ and fix any three distinct vertices $v_1, v_2, v_3$ in $V(G)$. Let $W$ be  an arbitrary set of size $\beta n$ excluding vertices $v_1,v_2,v_3$ and set
$V':=V(G)\setminus (W\cup\{v_1, v_2, v_3\})$. Define $X:=\{v\in V': w(v, v_1v_2v_3)>3t\}$.
Since $\de^w(G)\ge (\frac{1+3t}{4}+\mu)n$, Proposition~\ref{prop: heavy triangle}~\ref{quarter1} implies 
$|X|\ge \frac{1}{4}n-\be n$. Choose $v_4\in X$ such that $$w(v_4, v_1v_2v_3)=\max\{w(v,v_1v_2v_3):v\in X\}.$$ 
It follows from Proposition~\ref{prop: heavy triangle}~\ref{quarter2} that $|\{v\in V'\setminus\{v_4\}:w(v_1v_2v_3v_4,v)>1+3t\}|\geq \mu n-\be n$. Recall that $\beta< \frac{\mu}{2}$, there exists a vertex $v_5\in V'\setminus\{v_4\}$ such that  $w(v_1v_2v_3v_4,v_5)>1+3t$.

For convenience, we set $w(v_4v_5)=x$.
By the choice of $v_4$, we know that $3\geq w(v_1v_2v_3,v_4)\ge w(v_1v_2v_3,v_5)>1+3t-x$.
Therefore, $x>3t-2$ and 
\begin{align}\label{weight}
w(v_1v_2v_3,v_4v_5)=w(v_1v_2v_3,v_4)+w(v_1v_2v_3,v_5)>2(1+3t-x).
\end{align}
We claim that there are at least two heavy triangles among $v_1v_4v_5, v_2v_4v_5$ and $v_3v_4v_5$. Otherwise, 
at most one of  $v_1v_4v_5, v_2v_4v_5$ and $v_3v_4v_5$ is a heavy triangle in $G$. Then 
\[
w(v_1v_2v_3,v_4v_5)=\sum_{i\in [3]}w(v_i,v_4v_5)\le 2+2(3t-x),
\]
which contradicts with \eqref{weight}.

Assume, without loss of generality, that both $v_1v_4v_5$ and $v_2v_4v_5$ are heavy triangles.
Applying Proposition~\ref{prop: heavy triangle}~\ref{quarter2} again yields 
$$
|\{v\in V'\setminus \{v_4,v_5\}:w(v_1v_2v_4v_5,v)>1+3t\}|\geq \mu n-\be n-1>0\ \text{as}\  \beta<\frac{\mu}{2}.
$$
Choose $v_6\in V'\setminus  \{v_4,v_5\}$ such that  $w(v_1v_2v_4v_5,v_6)>1+3t$. 
Hence, both $v_1v_4v_5v_6$ and $v_2v_4v_5v_6$ are heavy $K_4$s in $G$. That is, $v_1$ and $v_2$ are $(K_4,\be n, 1)$-reachable as
there is a 3-set $S=\{v_4,v_5,v_6\}\subseteq V(G)\setminus W$ such that both $G[S\cup\{v_1\}]$ and $G[S\cup\{v_2\}]$ are heavy $K_4$'s, as required. 
\end{proof}
Our next lemma states that we can find a reachability partition with at most $2$ parts.

\begin{lemma}[{Reachability partition}]\label{lem: partition lem}
For any constants $\gamma,\beta',\mu>0$ with $\beta'<\frac{\mu}{2}$ and $t\in (0,1)$, there exist $\beta>0$ and $n_0\in \mathbb{N}$ such that the following holds for all $n\geq n_0$.
Let $G$ be an $n$-vertex weighted complete  graph such that {$\de^w(G)\ge (\frac{1+3t}{4}+\mu)n$} and every vertex in $V(G)$ is $(K_4,\beta' n, 1)$-reachable to at least $\gamma n$ other vertices. Then either  $V(G)$ is $(K_4,\beta n,2)$-closed, or there is a partition $\mathcal{P}=\{V_1, V_2\}$ of $V(G)$ such that $V_i$ is $(K_4,\beta n,4)$-closed and has size at least $\frac{\gamma}{2}n$ for every $i\in [2]$. 
\end{lemma}


In this case, we shall use the lattice-based absorbing method to show that $V_1\cup V_2$ is in fact closed. To illustrate this, we introduce some definitions that were originally introduced by Keevash and Mycroft~\cite{Keevash2015}, and we extend them to the edge-weighted setting.

\begin{defn}\label{lattice}
Let $G$ be an $n$-vertex  weighted graph and $\mathcal{P} = \{V_{1},\ldots,V_{k}\}$ be a vertex partition of $V(G)$ for some integer $k \ge 1$.
  \begin{enumerate}
     \item[{\rm (i)}] The \textit{index vector} $\mathbf{i}_{\mathcal{P}}(S) \in \mathbb{Z}^{k}$ of a subset $S \subseteq V(G)$ with respect to $\mathcal{P}$ is the vector whose $i$th coordinate is the size of the intersections of $S$ with $V_i$, i.e. $|S\cap V_i|$, for every $i\in[k]$.
     \item[{\rm (ii)}] Given an $h$-vertex graph $H$ and constants $\beta > 0$,  $t\in (0,1)$, an $h$-vector $\mathbf{i} \in \mathbb{Z}^{k}$ is $(H,\beta)$-\textit{robust} if for any set $W\subseteq V(G)$ of size $\be n$, the graph $G-W$ contains a $t$-heavy copy of $H$ whose vertex set has the index vector $\mathbf{i}$ with respect to $\mathcal{P}$.
     \item[{\rm (iii)}] Denote by $I_{\mathcal{P}}^{\beta}(G,H)$ the set of all $(H,\beta)$-robust $h$-vectors $\mathbf{i} \in \mathbb{Z}^{k}$. 
     \item[{\rm (iv)}] For every $j\in [k]$, let $\mathbf{u}_{j} \in \mathbb{Z}^{k}$ be the $j$th unit vector, i.e. $\mathbf{u}_j$ has $1$ on the $j$th coordinate and 0 on the other coordinates.
     \item[{\rm (v)}] A \textit{transferral} is a vector of the form $\mathbf{u}_{i}-\mathbf{u}_{j}$ for some distinct $i,j \in [k]$.
   \end{enumerate}
\end{defn}

The following lemma builds a sufficient condition that allows us to merge two distinct parts into a closed one, which can be proved by a similar argument as in~\cite{Han2024}.  We present its  proof in Appendix D of the arXiv version.

\begin{lemma}\label{lem: merge lem}
Given any positive integers $s,r$ with $r\ge3$ and constants $\beta>0,\,t\in (0,1)$, {the following holds for sufficiently large $n$}.
Let $G$ be an $n$-vertex weighted graph, and let $\mathcal{P}=\{V_1, \dots, V_C\}$ be a partition of $V(G)$ such that each $V_i$ is $(K_r, \beta n,s)$-closed. For distinct $i,j\in [C]$, if there exist two $r$-vectors $\mathbf{s}, \mathbf{t} \in I_{\mathcal{P}}^{\beta}(G,K_r)$ with $\mathbf{s-t}=\mathbf{u}_i-\mathbf{u}_j$, then $V_i\cup V_j$ is $(K_r,\frac{\beta}{2}n,2rs)$-closed.
\end{lemma}

To end this section, we give a full proof of Lemma~\ref{lem: partition lem}.
\begin{proof}[\textbf{Proof of Lemma~\ref{lem: partition lem}}]
Given $t\in (0,1)$, and $\gamma,\beta',\mu>0$ with $\beta'<\frac{\mu}{2}$, we further  choose constants $$\frac{1}{n}\ll\beta\ll\beta_1\ll\beta',\gamma.$$ 
Assume that $G$ is an $n$-vertex weighted complete  graph such that {$\de^w(G)\ge (\frac{1+3t}{4}+\mu)n$} and every vertex in $V(G)$ is $(K_4,\beta' n, 1)$-reachable to at least $\gamma n$ other vertices. We first have the following subtle observation to aid our proof.

\begin{fact}\label{fact: 2-reachable}
Given positive integers $m_1,m_2,s$ and a constant $t\in (0,1)$, let $H$ be an $n$-vertex weighted  graph. For any two vertices $u,w\in V(H)$, if there exist $m_1$ vertices that are $(K_4,m_2,s)$-reachable to both $u$ and $w$, 
then $u$ and $w$ are $(K_4,m,2s)$-reachable in $H$, where $m=\min\{m_1-1,m_2-4s\}$.
\end{fact}

We may assume that there are two vertices in $G$ that are not $(K_4,\beta_1 n,2)$-reachable to each other, say $v_1,v_2$, otherwise $V(G)$ is $(K_4,\beta_1 n,2)$-closed as desired. 
Also note that $v_1$ and $v_2$ are not $(K_4,\beta' n,1)$-reachable as $\beta_1\ll \beta'$.
Lemma \ref{lem: two reachable from three} tells that every three vertices of $V(G)$ must contain two vertices that are $(K_4,\beta' n,1)$-reachable. Hence, for every vertex $v\in V(G)\setminus \{v_1,v_2\}$, either $v$ and $v_1$ are $(K_4,\beta' n,1)$-reachable or $v$ and $v_2$ are $(K_4,\beta' n,1)$-reachable.

For each $v_i$ with $i\in[2]$, we write $\tilde{N}(v_i)$ for the set of vertices which are $(K_4,\beta'n,1)$-reachable to $v_i$. Then the following properties hold: 
\begin{enumerate}[label = (\roman{enumi})]
    \item \label{obs 1}  any vertex $v\in V(G)\setminus \{v_1,v_2\}$ lies in $\tilde{N}(v_1)\cup \tilde{N}(v_2)$.
    \item \label{obs 2} $|\tilde{N}(v_1)\cap \tilde{N}(v_2)|\le \beta_1n$. Otherwise, Fact~\ref{fact: 2-reachable} implies that $v_1,v_2$ are $(K_4,\beta_1n,2)$-reachable, a contradiction.
\end{enumerate}
For each $i\in [2]$, let 
\[
U_i=(\tilde{N}(v_i)\cup \{v_i\})\setminus \tilde{N}(v_{3-i}).
\]
We claim that $U_i$ is $(K_4,\beta'n,1)$-closed.
Otherwise, there exist two vertices $u_1,u_2\in U_i$ that are not $(K_4,\beta'n,1)$-reachable.
Since every vertex in $U_i$ is not $(K_4,\beta'n,1)$-reachable to $v_{3-i}$, we have that $\{u_1,u_2,v_{3-i}\}$ are not $(K_4,\beta'n,1)$-reachable to each other, also a contradiction.

Let $U_0=V(G)\setminus(U_1\cup U_2)$.
We have $|U_0|\le\beta_1n$ due to \ref{obs 1} and \ref{obs 2} above.
To construct the desired reachability partition, we proceed by reassigning each vertex from $U_0$ to either 
$U_1$ or $U_2$ in the following. 
By the assumption that each vertex $v\in U_0$ is $(K_4,\beta' n, 1)$-reachable to at least $\gamma n$ other vertices, we have  
$|\tilde{N}(v)\setminus U_0|\ge \gamma n-|U_0|\ge \gamma n-\beta_1n>2\beta_1n$ as $\beta_1\ll\gamma$.
Then for every $v\in U_0$, by the Pigeonhole principle, there exists some $i\in [2]$ such that $v$ is $(K_4,\beta' n, 1)$-reachable to at least $\beta_1n+1$ vertices in $U_i$.
Together with  $\beta_1\ll\beta'$, Fact~\ref{fact: 2-reachable} implies that $v$ is $(K_4,\beta_1 n, 2)$-reachable to every vertex in $U_i$.
Now partition $U_0$ as $U_0=R_1\cup R_2$, where for each $i\in [2]$, $R_i$ denotes a subset $U_0$ each of which is  $(K_4,\beta_1 n, 2)$-reachable to every vertex in $U_i$.
Applying Fact~\ref{fact: 2-reachable} again yields that every two vertices in $R_i$ are $(K_4,\beta n, 4)$-reachable for every $i\in [2]$ as $\beta\ll\beta_1$.
Let $V_i=R_i\cup U_i$ for each $i\in [2]$ and let $\mathcal{P}=\{V_1,V_2\}$.
Clearly, both $V_1$ and $V_2$ are $(K_4,\beta n, 4)$-closed.
Also, for each $i\in [2]$, as $\beta_1\ll\gamma$, one has 
\[
|V_i|\ge |U_i|\ge |\tilde{N}(v_i)|-\beta_1 n\ge \gamma n-\beta_1n\ge \frac{\gamma}{2}n.
\]
This completes the proof. 
\end{proof}


\subsection{Proof of Lemma~\ref{main lemma}}
\begin{proof}[{\textbf{Proof of Lemma~\ref{main lemma}}}]
Given constants $\mu>0$ and $t\in (0,1)$, 
we further choose
\[
\tfrac{1}{n}\ll \beta, \tfrac{1}{k}\ll\varepsilon\ll\tfrac{1}{p},\beta_1\ll \gamma\ll \beta'\ll \mu,
\]
and fix an $n$-vertex $p$-edge-colored weighted complete graph $G$ with $\de^w(G)\ge (\frac{1+3t}{4}+\mu)n$.

Assume there exists a vertex $v\in V(G)$ that is $(K_4,  \beta'n,1)$-reachable to a set $B_1$ of vertices with $|B_1|\leq \gamma n-1$. 
Let $B=B_1\cup \{v\}$ and $U=V(G)\setminus B$. 
We claim that $U$ is $(K_4,  (\beta'-\gamma) n,1)$-inner closed. 
For every two vertices $v_1,v_2\in U$, it is obvious that $v$ is not $(K_4,  \beta'n,1)$-reachable to $v_1$ or $v_2$ in $G$. It follows from Lemma~\ref{lem: two reachable from three} that $v_1$ and $v_2$ are $(K_4,\beta' n,1)$-reachable in $G$.
Therefore, by definition $U=V(G)\setminus B$ is $(K_4,(\beta'-\gamma)n,1)$-inner closed. As $\beta\ll \beta'-\gamma$, the set $U$ is $(K_4,\beta n,1)$-inner closed, as desired.

Next, we may assume that every vertex in $G$ is $(K_4,  \beta'n,1)$-reachable to at least $\gamma n$ vertices and  we shall prove that $V(G)$ itself is $(K_4,\beta n,32)$-closed. 
By applying Lemma~\ref{lem: partition lem} on $G$ with $\beta', \gamma$, we know  that either $V(G)$ is $(K_4,\beta_1n,2)$-closed or $V(G)$ admits a partition $\mathcal{P}=\{V_1,V_2\}$ where each $V_i$ is $(K_4,\beta_1 n,4)$-closed and $|V_i|\ge\frac{\gamma n}{2}$. In the following, we only consider the latter case.

Applying Lemma~\ref{thm: degree form of RL} on $G$ with the corresponding edge coloring $c(w,p)$ and ${\mathbf{d}:=(\frac{\mu}{8p},\cdots,\frac{\mu}{8p})}$, we obtain an $\eps$-regular partition $\mathcal{P}'=\{V_0\}\cup\{V_{i,j}: i\in [2], j\in [k_i]\}$ refining $\mathcal{P}$ and a spanning subgraph $G'\subseteq G$ with the properties~\ref{pro1}-\ref{pro5}, where we denote $k:=k_1+k_2$.
Let $R:=R(\eps)$ be the weighted reduced graph on the vertex set $\{V_{i,j}: i\in [2], j\in [k_i]\}$ and edge weight $w_R(V_{i,i'}V_{j,j'})=\sum_{\ell\in [p]}\frac{\ell-1}{p}\cdot d_\ell(V_{i,i'},V_{j,j'})$ for any $i,j\in [2],\,i'\in [k_i]$ and $j'\in [k_j]$. Then together with the choice $\eps\ll \tfrac{1}{p}\ll \mu$, Lemma~\ref{lem: degree of reduced graph} implies that
\begin{align}\label{min degree}
\delta^{w}(R)\ge \left(\frac{1+3t}{4}+\frac{\mu}{4}\right)k.
\end{align}


For each $i\in [2]$, set $\mathcal{V}_i:=\{V_{i,j}:j\in [k_i]\}$ and $\mathcal{P}^*:=\{\mathcal{V}_1,\mathcal{V}_2\}$. In what follows, we shall find transferrals in $R$ and then applying Lemma~\ref{lem: embedding lem} to obtain many vertex-disjoint heavy $K_4$ in $G$ with the corresponding index. Finally, we use Lemma~\ref{lem: merge lem} with $r=s=4$ to deduce that $V(G)=V_1\cup V_2$ is $(K_4,{\beta} n,32)$-closed. 
The following claim is an immediate consequence of Lemma \ref{lem: embedding lem}.
\begin{claim}\label{reduced-transferral}
    If $R$ contains two heavy $K_4$'s with vertex sets $S,\,T$ respectively, and ${\bf i}_{\mathcal{P}^*}(S)-{\bf i}_{\mathcal{P}^*}(T)=\mathbf{u}_1-\mathbf{u}_2$, then there are two $4$-vectors $\mathbf{s},\,\mathbf{t}\in I_{\mathcal{P}}^{2\beta}(G,K_4)$
    and $\mathbf{s-t}$ is a transferral of the form $\mathbf{u}_1-\mathbf{u}_2$. 
\end{claim}

Based on this result, our construction of transferrals essentially proceeds by detecting the existence of crossing heavy edges or triangles in $R$ in subsequent claims. Here a subgraph of $R$ is called \textit{crossing} if its vertex set intersects with both $\mathcal{V}_1$ and $\mathcal{V}_2$. 
For each $i\in [2]$, we call $\mathcal{V}_i$ a \textit{small} part in $\mathcal{P}^*$ if $k_i<\frac{1}{4}k$.

\begin{claim}\label{claim: no small partition}
If there exists a small part in $\mathcal{P}^*$, then $V(G)$ is $(K_4,{\beta}n,32)$-closed.
\end{claim}

\begin{proof}
Without loss of generality, assume that $\mathcal{V}_1$ is a small part in $\mathcal{P}^*$, that is, $k_1<\frac{1}{4}k$. 
Based on Lemma~\ref{lem: merge lem} and Claim~\ref{reduced-transferral}, it suffices to find two heavy $K_4$'s in $R$ with vertex sets $S,\,T$ respectively, such that ${\bf i}_{\mathcal{P}^*}(S)-{\bf i}_{\mathcal{P}^*}(T)=\mathbf{u}_1-\mathbf{u}_2$. 

Notice that $R$ is a weighted complete graph. Since $\delta^w(R)\ge \left(\frac{1+3t}{4}+\frac{\mu}{4}\right)k$, Proposition~\ref{prop: heavy triangle}~\ref{quarter1} implies that every vertex in $V(R)$ is incident to at least $\frac{k}{4}$ heavy edges. 
As $k_1<\frac{k}{4}$, there is a {crossing} heavy edge, say $V_{1,1}V_{2,1}$ in $R$ where $V_{1,1}\in \mathcal{V}_1$ and $V_{2,1}\in \mathcal{V}_2$. 
Applying  Proposition~\ref{prop: heavy triangle}~\ref{quarter1} again yields that for at least $\frac{k}{4}$ vertices $Z\in V(R)$, the triple $\{V_{1,1},V_{2,1},Z\}$ induces a 
heavy triangle in $R$. 
Since $k_1<\frac{1}{4}k$, there is a vertex $V_{2,2}\in \mathcal{V}_2$ such that $V_{1,1}V_{2,1}V_{2,2}$ is a heavy triangle in $R$. 
We apply Proposition~\ref{prop: heavy triangle}~\ref{quarter1} on the heavy triangle $V_{1,1}V_{2,1}V_{2,2}$ to pick a vertex $V_{2,3}\in \mathcal{V}_2$ such that $V_{1,1}V_{2,1}V_{2,2}V_{2,3}$ is a heavy $K_4$ in $R$ as $k_1<\frac{1}{4}k$. 

Similarly, $R$ contains a heavy edge with both vertices in $\mathcal{V}_2$ as $k_1<\frac{1}{4}k$, and by repeatedly applying Proposition~\ref{prop: heavy triangle}~\ref{quarter1}, one can obtain a heavy $K_4$ whose vertices belong to $\mathcal{V}_2$.
This together with $V_{1,1}V_{2,1}V_{2,2}V_{2,3}$ yields two heavy $K_4$'s in $R$ with index vectors $(0,4)$ and $(1,3)$, respectively, as desired.
\end{proof}
Next we may assume that neither $\mathcal{V}_1$ nor $\mathcal{V}_2$ is a  small part with respect to $\mathcal{P}^*$.

\begin{claim}\label{claim: no crossing heavy triangle}
If $R$ contains a crossing heavy triangle, then $V(G)$ is $(K_4,{\beta}n,32)$-closed.
\end{claim}
\begin{proof}
Without loss of generality, assume that $V_{1,1}V_{1,2}V_{2,1}$ is a heavy triangle in $R$. Then by Lemma~\ref{lem: merge lem} and Claim~\ref{reduced-transferral}, it suffices to find two heavy $K_4$'s with vertex sets $S,\,T$ respectively, such that ${\bf i}_{\mathcal{P}^*}(S)-{\bf i}_{\mathcal{P}^*}(T)=\mathbf{u}_1-\mathbf{u}_2$.
By applying Proposition~\ref{prop: heavy triangle}~\ref{quarter1}, we obtain a subset $\mathcal{Y}\subseteq V(R)$ of size at least $\frac{1}{4}k$ such that every vertex $Y\in \mathcal{Y}$ satisfies $w_R(Y,V_{1,1}V_{1,2}V_{2,1})>3t$. Choose $Y_1\in \mathcal{Y}$ such that $w_R(Y_1,V_{1,1}V_{1,2}V_{2,1})=\max\{w_R(Y,V_{1,1}V_{1,2}V_{2,1}):Y\in \mathcal{Y}\}$. 
Also there is a vertex $Y_2\in V(R)$ such that $w_R(Y_2,V_{1,1}V_{1,2}V_{2,1}Y_1)> 1+3t$, and such a vertex $Y_2$ has at least $\frac{\mu}{4}k$ choices by Proposition~\ref{prop: heavy triangle}~\ref{quarter2}. 
We may further assume that $Y_1$ and $Y_2$ lie in the same part $\mathcal{V}_j$ for some $j\in [2]$. Otherwise, suppose $Y_1\in \mathcal{V}_1$ and $Y_2\in \mathcal{V}_2$,  then $V_{1,1}V_{1,2}V_{2,1}Y_1$ and $V_{1,1}V_{1,2}V_{2,1}Y_2$ are two heavy $K_4$'s in $R$ with respective index vectors $(3,1)$ and $(2,2)$, and we are done. By the same reason, we further assume that no two vertices lying in distinct parts of $\mathcal{P}^*$ are $(K_4,0,1)$-reachable in $R$. Furthermore, combined with Proposition~~\ref{prop: heavy triangle}~\ref{quarter2}, this implies that no two vertices from distinct parts in $\mathcal{P}^*$ are $(K_3,0,1)$-reachable in $R$ either.
We proceed by considering whether 
$Y_1$ belongs to $\mathcal{V}_1$ or $\mathcal{V}_2$.
\medskip
    
{\bf Case 1.} $Y_1,Y_2\in \mathcal{V}_1$. \medskip

\noindent In this case, $V_{1,1}V_{1,2}Y_1$ is not a heavy triangle, otherwise $Y_1$ and $V_{2,1}$ are $(K_3,0,1)$-reachable, contrary to our assumption.
We set $w_R(Y_1Y_2)=x$.
By the choice of $Y_1$, we know that $3\geq w_R(V_{1,1}V_{1,2}V_{2,1},Y_1)\ge w_R(V_{1,1}V_{1,2}V_{2,1},Y_2)>1+3t-x$.
Thus $x>3t-2$ and 
\begin{align}\label{weight1}
w_R(V_{1,1}V_{1,2}V_{2,1},Y_1Y_2)=w_R(V_{1,1}V_{1,2}V_{2,1},Y_1)+w_R(V_{1,1}V_{1,2}V_{2,1},Y_2)>2(1+3t-x).
\end{align}
We claim that there are at least two heavy triangles among $V_{1,1}Y_1Y_2, V_{1,2}Y_1Y_2$ and $V_{2,1}Y_1Y_2$, as otherwise 
\[
w_R(V_{1,1}V_{1,2}V_{2,1},Y_1Y_2)=w_R(V_{1,1},Y_1Y_2)+w_R(V_{1,2},Y_1Y_2)+w_R(V_{2,1},Y_1Y_2)\le 2+2(3t-x),
\]
which contradicts with \eqref{weight1}. If $V_{2,1}Y_1Y_2$ and $V_{1,1}Y_1Y_2$ (or $V_{1,2}Y_1Y_2$) are heavy triangles, then $V_{2,1}$ and $V_{1,1}$ (resp.~$V_{1,2}$) are $(K_3,0,1)$-reachable in $R$, a contradiction. Thus $V_{1,1}Y_1Y_2, V_{1,2}Y_1Y_2$ are heavy triangles.
Recall that $V_{1,1}V_{1,2}Y_1$ is not a heavy triangle. Combined with Proposition~\ref{prop: heavy}~\ref{per1}, this means either $V_{1,1}V_{2,1}Y_1$ or $V_{1,2}V_{2,1}Y_1$ is a heavy triangle. This implies that $Y_2$ and $V_{2,1}$ are actually $(K_3,0,1)$-reachable in either case, also a contradiction.\medskip



{\bf Case 2.} $Y_1,Y_2\in\mathcal{V}_2$.\medskip
    
\noindent By a similar argument as \textbf{Case} 1, we derive that $V_{1,1}V_{1,2}Y_1$ is another heavy triangle (than $V_{1,1}V_{1,2}V_{2,1}$) inside $V_{1,1}V_{1,2}V_{2,1}Y_1$ and furthermore \[ V_{1,1}Y_1Y_2, V_{1,2}Y_1Y_2\ \text{are}\ \text{heavy triangles}.\] This tells that $Y_2$ and $V_{1,2}$ are actually $(K_3,0,1)$-reachable, a contrary to our assumption.  

This completes the proof.
\end{proof}

From Claim~\ref{claim: no crossing heavy triangle}, we may assume that there is no crossing heavy triangles in $R$, or else we are done.

\begin{claim}\label{claim: weight of crossing edge}
The weight of each crossing edge in $R$ is at most $\frac{3t-1}{2}$.
\end{claim}
\begin{proof}
Assume, without loss of generality, that  $V_{1,1}V_{2,1}$ is an arbitrary crossing edge in $R$ and let $x:=w_R(V_{1,1}V_{2,1})$.
For any vertex $Y\in V(R)\setminus \{V_{1,1},V_{2,1}\}$, Claim~\ref{claim: no crossing heavy triangle} guarantees that the triangle $V_{1,1}V_{2,1}Y$ cannot be $t$-heavy. Thus, $w_R(Y,V_{1,1}V_{2,1})\le 3t-x$. 
Since $\delta^w(R)\ge \left(\frac{1+3t}{4}+\frac{\mu}{4}\right)k$, one has 
\[
2\cdot\left(\frac{1+3t}{4}+\frac{\mu}{4}\right)k-2x\le w_R(V_{1,1}V_{2,1},V(R)\setminus \{V_{1,1},V_{2,1}\})\le (3t-x)\cdot (k-2), 
\]
which implies $x\le \frac{3t-1}{2}$.
\end{proof}

Assuming $k_1\le\frac{k}{2}$,
we have 
\[
d_{R}^w(V_{1,1})
\le k_1+\frac{3t-1}{2}\cdot k_2
=\frac{3-3t}{2}\cdot k_1+\frac{3t-1}{2}\cdot k
\le \frac{1+3t}{4}\cdot k,
\]
where the first inequality follows from Claim~\ref{claim: weight of crossing edge} while the last inequality holds as $k_1\le\frac{k}{2}$.
This yields a contradiction with (\ref{min degree}).

This completes the proof. 
\end{proof}

\section{Concluding remarks}
In this paper, we prove Conjecture~\ref{conj: main conj} for the tetrahedron case. However, for larger cliques $K_r$, new approaches are in demand.
Indeed, a crucial component of our proof depends on the existence of a reachable pair among every three vertices,  which implies that most vertices are $1$-reachable to many other vertices.
For $r=5$, the main technical barrier in our approach concerns the inability to guarantee the existence of at least one reachable pair for any finite set of vertices under a weaker weighted degree condition $\de^w(G)\ge (\frac{1+4t}{5}+o(1))n$.
On the other hand, establishing an almost cover via our augmentation technique is considerably complicated for larger $r$.
We hope that our new perspectives in Section 2.3 will be useful in the further study of Conjecture~\ref{conj: main conj}.

\bibliographystyle{abbrv}
\bibliography{ref.bib}

\begin{appendices}
\section{Proof of Lemma~\ref{thm: degree form of RL}}
We use the following fact to prove Lemma~\ref{thm: degree form of RL}.
\begin{fact}\label{fact: regular pair}
Given constants $\eta>\eps>0$ and a bipartite graph $G=X\cup Y$, if $(X,Y)$ is $\eps$-regular, then for all $X_1\subseteq X$ and $Y_1\subseteq Y$ with $|X_1|\ge \eta|X|$ and $|Y_1|\ge \eta|Y|$, we have that $(X_1,Y_1)$ is $\eps'$-regular in $G$ for any $\eps'\ge \max\{\frac{\eps}{\eta},2\eps\}$.
\end{fact}
\begin{proof}[{\textbf {Proof of Lemma~\ref{thm: degree form of RL}}}]
Given an integer $p\ge 2$, a constant $\eps>0$ and a real vector $\mathbf{d}=(d_1,\cdots,d_p)\in [0,1]^p$, we choose a positive constant $\eps'$ satisfying $\eps'\le \frac{\eps^2}{4p}$. 
Applying Theorem~\ref{thm: original form of RL} with parameter $\eps'$, there exist an integer $M'$ such that if we let $M:=\frac{4M'}{\eps}\geq M'$ and $G$ is a $p$-edge-colored graph on $n\geq M$ vertices, then $G$ has  an $\eps'$-regular partition $\mathcal{P}'=\{V_0',V_1',\cdots,V_{k'}'\}$ of the vertex set $V$ with $|V_1'|=\cdots=|V_{k'}'|=:m'$ where $\frac{1}{\eps'}\le k'\le M'$.
We will construct a  spanning subgraph $G'$ of $G$ together with a partition $\mathcal{P}=\{V_0,V_1,\cdots,V_k\}$ of $V(G)$ for some $k\le M$ satisfying properties~\ref{pro1}-\ref{pro5} by carrying out the following edge-removal procedures. 

(i) Remove the edges of non-regular pairs. Let $S_1$ be a set consisting of triples $(V_i',V_j',\ell)$ such that $(V_i',V_j')$ is not $\eps'$-regular in $G_{\ell}$. We recolor all edges between $V_i'$ and $V_j'$ in $G_{\ell}$ red if $(V_i',V_j',\ell)\in S_1$. 
For any vertex $v\in V$, if there are at least $\frac{\eps}{4} n$ red edges incident to $v$, then we move $v$ to $V_0'$ and delete all red edges that do not incident to any vertex in the up-to-date $V_0'$.
After deleting these edges, we observe that the degree of any vertex $v\in V$ in graph $G_{\ell}$ is at least $d_{G_{\ell}}(v)-\frac{\eps}{4}n$ for each $\ell\in [p]$.
As $\eps'\le \frac{\eps^2}{4p}$ and all but at most $\eps'k'^2$ pairs $(V_i',V_j')$ are $\eps'$-regular in $G$, the number of vertices we have moved to $V_0'$ is at most 
\[
\frac{\eps'k'^2\cdot m'^2}{\frac{\eps}{4}n}\le \frac{\eps}{4}n.
\]

(ii) Remove the edges of regular pairs with low density. Let $S_2$ be a set consisting of triples $(V_i',V_j',\ell)$ such that $d_{\ell}(V_i',V_j')\le d_{\ell}+\eps'$. We recolor all edges between $V_i'$ and $V_j'$ in $G_{\ell}$ blue for  every $(V_i',V_j',\ell)\in S_2$.
For $i,j\in [k']$ with $i\neq j$ and $\ell\in [p]$, let $X_{i,j,\ell}$ be a subset of $V_i'$ such that every vertex in $X_{i,j,\ell}$ is incident to at least $(d_{\ell}+2\eps')m'$ blue edges with the other endpoint in $V_j'$ in graph $G_{\ell}$. 
Note that $|X_{i,j,\ell}|\le \eps'm'$ by Fact~\ref{fact: large deg}.
For each $v\in X_{i,j,\ell}$, we mark all but $(d_{\ell}+2\eps')m'$ of the blue edges incident to $v$ in $G_{\ell}$. Clearly,  the total number of marked edges is at most
$\binom{k'}{2}\cdot p\cdot \eps'm'\cdot m'\le \frac{p\eps'n^2}{2}.$
For every vertex $v\in V$, if it is adjacent to at least $\frac{\eps}{4}n$ marked edges, then move $v$ to $V_0'$ and delete all blue edges that do not incident to any vertex in the updated $V_0'$.
Now, in each $G_{\ell}$, for every vertex $v\in V$, we delete at most $(d_{\ell}+2\eps')m'\cdot k'+\frac{\eps}{4}n\le (d_{\ell}+\frac{\eps}{2})n$ edges incident to $v$.
As $\eps'\le \frac{\eps^2}{4p}$, the number of vertices we have moved to $V_0'$  in this stage is at most
\[
\frac{\frac{p\eps'n^2}{2}}{\frac{\eps}{2}n}\le\frac{\eps}{4}n.
\]
After the above process, 
we observe that for each $\ell\in [p]$, $(V_i',V_j')$ is $\eps'$-regular with density either $0$ or at least $d_{\ell}+\eps'$.

(iii) Delete all edges in $G_{\ell}[V_i']$ for all $i\in k'$ and $\ell\in [p]$. 
Note that $|V_i'|\le \eps'n$ as $\frac{1}{\eps'}\le k'$.
Hence, for every $v\in V$, at most $\eps'n$ edges incident to $v$ are deleted in this step.

(v) Finally, to guarantee that all clusters have the same size, we further divide them into smaller blocks each of size $\lceil \frac{\eps n}{4k'}\rceil$.
Move the vertices that are left over in each
cluster after this process into the exceptional set $V_0'$.
Call this new exceptional set $V_0$ and the other blocks $V_1,\ldots,V_k$. 

Denote  by $G'$ the resulting spanning subgraph after removing. It is easy to check that $G'$ together with the vertex partition $\mathcal{P}=\{V_0,V_1,\ldots,V_k\}$ satisfies properties~\ref{pro1}, \ref{pro2}, \ref{pro4} and \ref{pro5} by Fact~\ref{fact: regular pair}. 
For every $v\in V$ and $\ell\in [p]$, we have
\[
d_{G_{\ell}'}(v)\ge d_{G_{\ell}}(v)-\frac{\eps}{4}n-(d_{\ell}+\frac{\eps}{2})n-\eps'n\ge d_{G_{\ell}}(v)-(d_{\ell}+\eps)n.
\]
This completes the proof.
\end{proof}

\section{Proof of Lemma~\ref{lem: abs set}}
The proof of Lemma~\ref{lem: abs set} can be easily derived by the following result. We first give the definition of absorbers for hypergraphs. Let $F$ be a $k$-graph with $b$ vertices. Suppose $\mathcal{G}$ is a $k$-graph with vertex set $V$ and let $S\subseteq V$ with $|S| = b$ be given. We call $\emptyset\neq A\subseteq V\setminus S$  an $(F,s)$-absorber,  if $|A|\leq b^2s$  and both $\mathcal{G}[A]$ and $\mathcal{G}[A\cup S]$ contain $F$-factors.

\begin{lemma}[\cite{chang2020factors}]\label{chang}
Let $F$ be a $k$-graph with $b$ vertices. Let $\gamma>0, 0<\varepsilon<\min\{1/3, \gamma/2\}$ and $s\in \N$.
  Then there exists $\xi>0$ such that the following holds for every sufficiently large~$n$.
  Suppose $\mathcal{G}$ is an $n$-vertex $k$-graph such that, there exists $V_0\subset V(\mathcal{G})$ of size at most $\varepsilon n$ such that for every $b$-subset $S$ of $V(\mathcal{G})\setminus V_0$, there are at least $\gamma n$ vertex-disjoint $(F,s)$-absorbers for $S$, and for every vertex $v\in V(\mathcal{G})$, there are at least $\gamma n$ copies of $F$ containing $v$, where each pair of these copies only intersects at vertex $v$.
  Then $\mathcal{G}$ contains a subset $A\subseteq V(\mathcal{G})$ of size at most $\gamma n$ such that, for every subset
  $R\subseteq V(\mathcal{G})\setminus A$ with $|R|\leq \xi n$ such that $b$ divides $|A|+|R|$, the $k$-graph $\mathcal{G}[A\cup R]$ contains an
  $F$-factor.
\end{lemma}
\begin{proof}[{\textbf{Proof of Lemma~\ref{lem: abs set}}}]

Given $s\in \mathbb{N}$, $t\in(0,1)$ and $\gamma>0$, we shall choose $\frac{1}{n}\ll\xi\ll \frac{1}{s},\gamma$. Let $G$ be an edge-weighted graph with $|V(G)|=n$ such that for every $S\in {V(G)\choose 4}$ there is a family of at least $\gamma n$ vertex-disjoint $(K_4,s)$-absorbers. We construct a $4$-graph $\mathcal{G}$ on vertex set $V(G)$ where $E(\mathcal{G})$ consists of every $t$-heavy copy of $K_4$ in $G$. Let $F$ be a $4$-graph with $4$ vertices and exactly one edge. Then by definition, every $S\in {V(\mathcal{G})\choose 4}$ has at least $\gamma n$ vertex-disjoint $(F,s)$-absorbers. By applying Lemma~\ref{chang} on $\mathcal{G}$ with $b=4$ and $V_0=\emptyset$, we obtain a subset $A\subseteq V(\mathcal{G})$ of size at most $\gamma n$ such that, for every subset
  $R\subseteq V(\mathcal{G})\setminus A$ with $|R|\leq \xi n$ such that $4$ divides $|A|+|R|$, the $4$-graph $\mathcal{G}[A\cup R]$ contains an $F$-factor.
Then $A$ is a desired $\xi$-absorbing set in $G$.
\end{proof}

\section{Proof of Lemma~\ref{lem: absorber lem}}
Recall that two vertices $u,v$ are $(K_r,\beta n,s)$-reachable if for any set $W\subseteq V(G)$ of size at most $\beta n$, there exists a set  $S\subseteq V(G)\setminus W$ of size at most $rs-1$ such that both $G[S\cup \{u\}]$ and $G[S\cup \{v\}]$ have heavy $K_r$-factors, where we call such $S$ a \textit{$K_r$-connector} for $u,v$.
\begin{proof}[{\textbf{Proof of Lemma~\ref{lem: absorber lem}}}]
For any constants $s\in \mathbb{N}$, $\beta>0$ and $t\in (0,1)$, let $G$ be a weighted graph such that $V(G)$ is $(K_4,  \beta n,s)$-closed.
For any $4$-subset $S=\{v_1,v_2,v_3,v_4\}\subseteq V(G)$, our goal is to construct as many pairwise vertex-disjoint $(K_4,s)$-absorbers for $S$ as possible.
Let $\mathcal{A}=\{A_1,A_2,\cdots,A_{\ell}\}$ be a maximal family of vertex-disjoint $(K_4,s)$-absorbers constructed so far. Suppose, to  the contrary, that $\ell<\frac{\beta}{64s}n$. We have that the number of vertices covered by $\mathcal{A}$ is at most $\frac{\beta}{64s}n\cdot 16s=\frac{\beta}{4}n$ as each $A_i$ has size at most $16s$.

Since $V(G)$ is $(K_4,\beta n,s)$-closed, we can find a heavy $K_4$ in $V(G)\setminus (\cup_{i=1}^{\ell}A_i\cup S)$, whose vertex set is denoted as $T=\{u_1,u_2,u_3,u_4\}$.
Next, we greedily choose a family $\{S_1,S_2,S_3,S_4\}$ of vertex-disjoint subsets in $V(G)\setminus (\cup_{i=1}^{\ell}A_i\cup S\cup T)$ such that $S_i$ is a $K_4$-connector for $u_i,v_i$ with $|S_i|\le 4s-1$ for each $i\in [4]$.
This can be done as $|\cup_{i=1}^{\ell}A_i\cup (\cup_{i=1}^4 S_i)\cup S\cup T|\le \beta n$ and every pair of $u_i$ and $v_i$ are $(K_4,\beta n,s)$-reachable.
It is easy to verify that $(\cup_{i=1}^4S_i)\cup T$ is also a $(K_4,s)$-absorber for $S$, contrary to the maximality of $\ell$.
\end{proof}

\section{Proof of Lemma~\ref{lem: merge lem}}
\begin{proof}[{\textbf{Proof of Lemma~\ref{lem: merge lem}}}]
Given constants  $\be>0$, $t\in (0,1)$, let $G$ be an $n$-vertex weighted graph and $\mathcal{P}=\{V_1, V_2,\ldots,V_C\}$ be a partition of $V(G)$ as in the assumption.
Moreover let $i\neq j\in [C]$ be such that there exist $r$-vectors $\textbf{s},\textbf{t}\in I^{\be}_{\mathcal{P}}(G,K_r)$ such that $\textbf{s}-\textbf{t}=\textbf{u}_i-\textbf{u}_j$. Without loss of generality (relabelling if necessary) we may assume that $i = 1$ and $j = 2$. Then we may write $\textbf{s}=(s_1,s_2,\ldots, s_C)$ and  $\textbf{t}=(s_1-1,s_2+1,s_3,\ldots, s_C)$ for some $s_i\in \N$ such that $\sum_{i=1}^Cs_i=r$.  It suffices to show that every two vertices $x\in V_1$ and $y\in V_2$ are $\left(
K_r,\tfrac{\be}2n,2rs\right)$-reachable so let us fix such an $x$ and $y$. Let $W\subset V(G)\setminus \{x,y\}$ be a vertex subset of size at most $ \tfrac{\be}{2}n$. By the assumption of $(K_r,\be)$-robustness, we can pick two vertex-disjoint copies of heavy 
$K_r$ in $V(G)\setminus (W\cup\{x,y\})$ whose corresponding vertex sets $S, T$ have index vectors $(s_1,s_2,\ldots, s_C)$ and $(s_1-1,s_2+1,\ldots, s_C)$, respectively.\medskip

Note that we may choose $x'\in S\cap V_1$ and $y'\in T\cap V_2$ such that by letting $S\setminus \{x'\}=\{u_1,u_2,\ldots, u_{r-1}\}$ and $T\setminus \{y'\}=\{v_1,v_2,\ldots, v_{r-1}\}$, $u_j$ and $v_j$ belong to the same part of $\mathcal{P}$ for each $j\in [r-1]$. Since each $V_i$ is $(K_r,\be n, s)$-closed for each $i\in [C]$, we greedily pick a collection $\{S_1,S_2,\ldots, S_{r-1}\}$ of vertex-disjoint subsets in $V(G)\setminus (W\cup S\cup T\cup\{x,y\})$ such that each $S_j$ is a $K_r$-connector for $u_j,v_j$ with $|S_j|\le rs-1$. Indeed, note that as $n$ sufficiently large we have that  $|W|\le \be n- 2r^2s$ and so for any $r'\leq r-1$ we have
\[\left|\left(\bigcup_{j=1}^{r'}S_j\right)\cup W\cup S\cup T\cup\{x,y\}\right|\le \be n.\] We can therefore indeed pick the $S_j$ one by one because of the fact that $u_j$ and $v_j$ are $(K_r,\be n ,s)$-reachable. Similarly, we additionally choose two vertex-disjoint (from each other and all other previously chosen vertices) $K_r$-connectors, say $S_x$ and $S_y$, for $x,x'$ and $y,y'$, respectively. At this point, it is easy to verify that the subset $\hat S:=\bigcup_{i=1}^{r-1}S_i\cup S_x\cup S_y\cup S\cup T$ is actually a $K_r$-connector for $x,y$ with size at most $2r^2s-1$. For example if we want a heavy $K_r$-factor in $G[\hat S \cup \{x\}] $ (leaving $y$ uncovered), we can take the heavy $K_r$-facots in $G[S_x\cup\{x\}]$, $G[S_y\cup\{y'\}]$ and $G[S_j\cup\{v_j\}]$ for $j\in [r-1]$, as well as the copy of heavy $K_r$ on $S$.   Therefore by definition, $x$ and $y$ are $\left(K_r,\tfrac{\be}{2}n, 2rs\right)$-reachable.
\end{proof}
\end{appendices}

\end{document}